\documentclass[11pt]{amsart}
\usepackage[T1]{fontenc}
\usepackage{lmodern}
\usepackage{amsmath,amssymb,amsthm,mathtools}
\usepackage{microtype}
\usepackage{booktabs}
\usepackage{enumitem}
\usepackage[colorlinks=true,allcolors=blue]{hyperref}

\allowdisplaybreaks
\newtheorem{theorem}{Theorem}[section]
\newtheorem{proposition}[theorem]{Proposition}
\newtheorem{lemma}[theorem]{Lemma}
\newtheorem{corollary}[theorem]{Corollary}
\theoremstyle{definition}
\newtheorem{definition}[theorem]{Definition}
\newtheorem{example}[theorem]{Example}
\theoremstyle{remark}
\newtheorem{remark}[theorem]{Remark}

\newcommand{\R}{\mathbb{R}}
\newcommand{\abs}[1]{\lvert #1\rvert}
\newcommand{\norm}[1]{\lVert #1\rVert}
\newcommand{\cA}{\mathcal{A}}

\newcommand{\cQ}{\mathcal{Q}}

\newcommand{\cV}{\mathcal{V}}
\newcommand{\cW}{\mathcal{W}}
\newcommand{\ind}{\mathbf{1}}

\newcommand{\eps}{\varepsilon}
\newcommand{\Lfac}{\mathfrak L}
\newcommand{\disjunion}{\mathbin{\dot\bigcup}}

\DeclareMathOperator{\dist}{dist}
\DeclareMathOperator{\width}{width}
\DeclareMathOperator{\co}{co}

\title[Profile-Stable Multiplicity Factoring]{Profile-Stable Buffered
Multiplicity Factoring in Four Dimensions}
\author{Zhixu Hua}
\address{Changkong College, Nanjing University of Aeronautics and Astronautics,
Nanjing, Jiangsu, China}
\email{hzx001@nuaa.edu.cn}

\author{Xinshun Yao}
\address{College of General Aviation and Flight, Nanjing University of Aeronautics
and Astronautics, Liyang, Jiangsu, China}
\email{1586847002@QQ.COM}

\author{Xiufan Yang}
\address{School of Science, Nanjing University of Posts and Telecommunications,
Nanjing, Jiangsu, China}
\email{B25100020@njupt.edu.cn}
\date{}

\begin{document}

\begin{abstract}
Multiplicity factoring is usually formulated for child families at comparable
scales.  For children of mixed geometry, thickening at the shortest parent scale
produces nonuniform inflation ratios, and a single worst-case replacement does not
preserve the natural density normalization.

We prove a multiplicity-factoring theorem for finite indexed convex parent--child
families in $\R^4$ that accommodates arbitrary child shapes, scales, orientations,
aspect ratios, and repetitions.  The local geometry of each assigned family is
encoded by a thickening-weighted Frostman coefficient and a mean-normalized inflation
efficiency, both determined by the base family before any shading refinement.  The
resulting coarse density satisfies
\[
 \Lfac^{-A_\eps}(w_1/w_4)^\eps
 \mathfrak P_\eps\lambda^{K_\eps},
\]
up to the stated parameter-dependent constant, where $\mathfrak P_\eps$ is an
explicit profile of the parent loads and local efficiencies.

The proof thickens arbitrary measurable shadings, projects along a shortest parent
direction, establishes an indexed three-dimensional convex-union estimate in the
presence of collisions, and lifts the resulting density response back to four
dimensions.  A weighted H\"older inequality then assembles the nonuniform parent
data, while a common cellular refinement regularizes the fine and coarse
multiplicities and yields the stated parentwise multiplicity-product estimate.

Under a relative convex Frostman hypothesis, the profile is expressed explicitly in
the minimum, mean, and maximum inflation ratios.  The comparable-scale regime follows
as a specialization after a preliminary load selection.  Thus the theorem supplies
a structural factoring input for four-dimensional overlap arguments; deriving new
Kakeya maximal or Hausdorff-dimension estimates would require additional analytic
ingredients.
\end{abstract}

\maketitle
\tableofcontents

\section{Introduction}

Multiplicity factoring separates two sources of overlap in a parent--child family:
the overlap among the parents and the overlap among the children assigned to one
parent.  This separation is a recurring feature of Kakeya arguments, from the
overlap methods of C\'ordoba and Wolff through Katz--\L aba--Tao, Katz--Tao, and the
multilinear polynomial method
\cite{cordoba1977kakeya,wolff1995improved,katzlabatao2000minkowski,
katztao2002newbounds,guth2010endpoint}; see also \cite{zahl2026survey}.  The convex-set
framework of Wang and Zahl and the factoring perspective of Guth, Wang, and Zahl
provide the three-dimensional background for the argument below
\cite{wangzahl2025convexunions,guthwangzahl2026streamlined}.

A comparable-scale formulation makes this factorization particularly transparent.
When the children have uniformly comparable John axes, thickening them at a fixed
radius changes their volumes by comparable factors, and a single local density can
be passed through the argument.  That normalization is too rigid once the assigned
children have genuinely different geometries.  At the shortest parent scale $h$, the
quantities
\[
 r_W(V)=\frac{\abs{N_h(V)}}{\abs V}
\]
may vary widely, even within one parent.  Replacing them by the largest inflation
ratio loses the mean information needed to compare the original and thickened
densities, while discarding the weights breaks the relative Frostman normalization.

The phenomenon is not limited to children of different volumes.  If $\tau\ll h$,
consider equal-volume bodies with John-axis quadruples
\[
 (\tau^3,\tau,\tau,\tau)
 \quad\hbox{and}\quad
 (\tau^2,\tau^2,\tau,\tau).
\]
Their coordinatewise axis ratios become unbounded as $\tau\to0$, but both
$h$-neighborhoods have volume comparable to $h^4$.  Thus their inflation ratios are
comparable although the usual coordinatewise hypothesis fails.  Example
\ref{ex:mixed-subbuffer} gives the precise calculation.  This model suggests that
the relevant local datum is the behavior under thickening, rather than a direct
comparison of the child axes.

We implement this idea by weighting each child with its inflation ratio.  For an
assigned family, the minimum, volume-weighted mean, and maximum inflation ratios are
denoted by
\[
 r_W^-\le \overline r_W\le r_W^+.
\]
The thickening-weighted relative Frostman coefficient
$C^{\rm wt}_{F,h}(\cV_W,W)$ measures how the weighted body mass can concentrate in a
convex subregion of the parent.  Together with the mean-normalized density conversion
it gives the local efficiency $\kappa_W$.  Under an ordinary relative Frostman
hypothesis, this efficiency is bounded below by the explicit quantity
\[
 \eta_W^{(K)}=
 \frac{\overline r_W}{r_W^+}
 \left(\frac{r_W^-}{\overline r_W}\right)^K.
\]
All of these are base-family data: they are determined before the shading is refined
and remain fixed throughout the construction.

Theorem~\ref{thm:main} applies this local normalization to finite indexed convex
parent--child families in $\R^4$.  The parents have comparable John dimensions, but
the children may have arbitrary shapes, scales, aspect ratios, orientations, and
indexed repetitions.  At the common analytic radius $h=w_1/100$, each parent is
replaced by the stable buffer $W^\sharp=N_{w_1}(W)$.  A single refinement then
regularizes the inner, global, and coarse multiplicities, satisfies the parentwise
multiplicity product, and produces compatible fine and coarse fibres.  The density of
the coarse shading has the form
\[
 \Lfac^{-A_\eps}\left(\frac{w_1}{w_4}\right)^\eps
 \mathfrak P_\eps\lambda^{K_\eps},
\]
up to the stated dependence on $\eps$ and the parent comparability constant.  Here
$\mathfrak P_\eps$ is an explicit profile of the local efficiencies and parent loads.

The load profile is needed because the main theorem keeps every parent with positive
initial shaded mass.  If $D_W=T_W/\abs W$ denotes the child load of a parent, the
weighted H\"older coefficient
\[
 \mathfrak P_K(\mathbf D,\boldsymbol\alpha;\mathbf p)
 =
 \frac{(\sum_Wp_WD_W)^K}
 {(\sum_Wp_WD_W^{K/(K-1)}
       \alpha_W^{-1/(K-1)})^{K-1}}
\]
assembles the nonuniform local estimates without first selecting a common load
level.  It records the actual obstruction created by the fixed coarse denominator;
when the loads and local efficiencies are balanced, it reduces to a positive
constant of the expected size.

The local argument has three geometric components.  First, arbitrary measurable
shadings are thickened with their exact inflation weights.  The resulting bodies are
projected along a shortest John direction of the parent, which makes the volume of
each thickened body comparable to $h$ times its projected volume without restricting
the orientations of the children.  The
projected family need not be essentially distinct: different indexed children may
produce the same three-dimensional body.  We therefore pass from the
essentially-distinct convex-union estimate of Wang and Zahl to an indexed estimate by
means of affine-adapted canonical boxes.  On each collision layer, the loss in
geometric body mass is canceled by the gain in indexed Katz--Tao density.  Finally, a
layer-cake and Jensen argument lifts the three-dimensional density response back to
four dimensions without changing its exponent.

The global argument combines these parentwise estimates through one cellular
refinement.  Two initial selections make the inner and global fine multiplicities
uniform.  A global null cleanup aligns pointwise membership with positive-measure
cell activity.  The remaining selections regularize the indexed active-parent count
and the fine mass of a cell.  Because every stage restricts all children by a common
union of half-open cells, the multiplicity conclusions, fibre relation, local union
uniformity, and density estimate hold for the same refinement.  The weighted
H\"older inequality then converts the sum of the local responses into the displayed
profile density.

Comparable child geometry reappears as a specialization rather than as an input to
the main theorem.  Corollary~\ref{cor:comparable-scale} first selects a dyadic parent-load
level; coordinatewise child-axis comparability then makes the inflation efficiencies
uniformly positive.  The general profile consequently collapses to the familiar
comparable-scale conclusion, with the corresponding selected coarse normalization.
The distinction between the two normalizations is essential when the original loads
are highly unbalanced.

The result is intended as a structural factoring theorem for four-dimensional
overlap arguments.  It does not by itself supply the restriction, two-ends,
transversality, or other analytic input needed for a new four-dimensional Kakeya
maximal or Hausdorff-dimension estimate.

Section~\ref{sec:prelim} introduces the indexed conventions and profile data and
states the main results.  Section~\ref{sec:lift} develops the codimension-one lift.
Section~\ref{sec:indexed} proves the indexed three-dimensional convex-union estimate,
and Section~\ref{sec:local-r4} converts it into the nonuniform local theorem in four
dimensions.  Section~\ref{sec:cellular} contains the H\"older assembly and cellular
construction.  Section~\ref{sec:profile-corollaries} treats balanced profiles and the
comparable-scale regime, while Section~\ref{sec:sharpness} explains the geometric and
normalization choices through model examples.  The appendices collect quantitative
dependence, the derivation of the external input, and supplementary boundary cases.

\section{Indexed families and the profile theorem}
\label{sec:prelim}

This section fixes the indexed conventions and the quantities used to describe the
geometry inside one parent.  The order is important: the body family determines the
inflation and concentration data first, and subsequent shading refinements are
measured against those fixed denominators.  We then state the main theorem in a form
that keeps the initial positive-mass parent support visible in the coarse
normalization.

\subsection{Convex bodies and indexed notation}

A convex body in $\R^n$ is compact, convex, nonempty, and has positive
$n$-dimensional Lebesgue measure.  If $K$ is a convex body, its John ellipsoid is the
maximal-volume ellipsoid $E_K\subset K$ \cite{john1948extremum}.  We normalize it by
\begin{equation}\label{eq:john-inclusions}
 E_K\subset K\subset nE_K.
\end{equation}
Its ordered semiaxis lengths are
$a_1(K)\le\cdots\le a_n(K)$.  A John box is aligned with the principal axes and has
side lengths comparable to these semiaxes.

All families are finite and indexed.  Thus
$\cA=(A_i)_{i\in I}$ may contain $A_i=A_j$ for $i\ne j$, and sums and
multiplicities count indices.  For a shading $Y(A_i)\subset A_i$, write
\begin{align*}
 U(\cA,Y)&=\bigcup_{i\in I}Y(A_i),
 &M(\cA,Y)&=\sum_{i\in I}\abs{Y(A_i)},\\
 T(\cA)&=\sum_{i\in I}\abs{A_i},
 &\lambda(\cA,Y)&=\frac{M(\cA,Y)}{T(\cA)},\\
 m_{\cA,Y}(x)&=\sum_{i\in I}\ind_{Y(A_i)}(x),
 &\mu(\cA,Y)&=\frac{M(\cA,Y)}{\abs{U(\cA,Y)}}.
\end{align*}
The last quantity is used only for positive-measure unions.  The fibre at $x$ is
\[
 (\cA)_Y(x)=\{i\in I:x\in Y(A_i)\}.
\]
A refinement satisfies $Y'(A_i)\subset Y(A_i)$ for every index.  The base family
continues to carry all of its indices, with empty refined shading on a discarded
index, except when a theorem first restricts to its positive-mass input support and
then explicitly extends the output back to the original family.

For a positive-volume convex test body $K$, define
\[
 \Delta(\cA,K)=\frac1{\abs K}
 \sum_{i:A_i\subset K}\abs{A_i},
 \qquad
 \Delta_{\max}(\cA)=\sup_K\Delta(\cA,K).
\]
If every $A_i\subset W$, the relative Frostman constant is
\begin{equation}\label{eq:frostman-def}
 C_F(\cA,W)=
 \sup_{\substack{K\subset W\\K\text{ a convex body}}}
 \frac{\Delta(\cA,K)}{\Delta(\cA,W)}.
\end{equation}
This is imposed only on nonempty families.  For $h>0$, let
$N_h(A)=\{x:\dist(x,A)\le h\}$, and put
$\log_+t=\max\{0,\log t\}$.

\subsection{Local profile data determined by the base family}

Let $\cV_W$ be a nonempty assigned indexed family of convex bodies $V\subset W$.
Thickening at radius $h$ changes the volume of each child by a factor that depends on
its shape and scale.  The quantities below retain the minimum, maximum, and
body-mass-weighted mean of these inflation factors, while $D_W$ records the total
child load in the parent.
Set
\[
 T_W=\sum_{V\in\cV_W}\abs V,\qquad
 D_W=\frac{T_W}{\abs W},\qquad
 r_W(V)=\frac{\abs{N_h(V)}}{\abs V},
\]
\begin{equation}\label{eq:inflation-statistics}
 r_W^-=\min_{V\in\cV_W}r_W(V),\qquad
 r_W^+=\max_{V\in\cV_W}r_W(V),\qquad
 \overline r_W=\frac{\sum_Vr_W(V)\abs V}{T_W}.
\end{equation}
Thus $0<r_W^-\le\overline r_W\le r_W^+<\infty$.  For any measurable
shading or later refinement $Y$, define
\begin{equation}\label{eq:local-densities}
 d_W(Y)=\frac{\sum_V\abs{Y(V)}}{T_W},\qquad
 d^{\rm wt}_{W,h}(Y)=
 \frac{\sum_Vr_W(V)\abs{Y(V)}}{\sum_Vr_W(V)\abs V}.
\end{equation}
Since $r_W(V)\ge r_W^-$,
\begin{equation}\label{eq:weighted-unweighted-density}
 d^{\rm wt}_{W,h}(Y)\ge
 \frac{r_W^-}{\overline r_W}\,d_W(Y).
\end{equation}
The denominator in the weighted density is the total thickened body mass.  This is
why the conversion in \eqref{eq:weighted-unweighted-density} involves the mean
$\overline r_W$ rather than the worst inflation $r_W^+$.

The same weights also alter the relative Frostman distribution.  The next
coefficient measures how much of the weighted body mass can concentrate in a convex
subregion of $W$, normalized by its density in the full parent.

\begin{definition}[Weighted relative Frostman coefficient]
\label{def:weighted-frostman}
Define
\begin{equation}\label{eq:weighted-frostman}
 C^{\rm wt}_{F,h}(\cV_W,W)
 =
 \sup_{\substack{K\subset W\\K\text{ a convex body}}}
 \frac{\displaystyle
  \abs K^{-1}\sum_{\substack{V\in\cV_W\\V\subset K}}
        r_W(V)\abs V}
 {\displaystyle
  \abs W^{-1}\sum_{V\in\cV_W}r_W(V)\abs V}.
\end{equation}
\end{definition}

Testing $K=W$ shows $C^{\rm wt}_{F,h}\ge1$.  This coefficient and all quantities
in \eqref{eq:inflation-statistics} depend only on the base family, not on a shading.
The two efficiencies below combine weighted concentration with the density conversion.
The first retains the exact weighted Frostman coefficient; the second is its explicit
counterpart under an ordinary relative Frostman hypothesis.
For $K>1$, put
\begin{align}
 \kappa_W^{(K)}
 &=
 \bigl(C^{\rm wt}_{F,h}(\cV_W,W)\bigr)^{-1}
 \left(\frac{r_W^-}{\overline r_W}\right)^K,
 \label{eq:kappa-def}\\
 \eta_W^{(K)}
 &=
 \frac{\overline r_W}{r_W^+}
 \left(\frac{r_W^-}{\overline r_W}\right)^K
 =
 \frac{(r_W^-)^K}{r_W^+(\overline r_W)^{K-1}}.
 \label{eq:eta-def}
\end{align}
Both numbers lie in $(0,1]$.

The local efficiencies must eventually be combined across parents whose loads may
be very different.  For positive $D_W,\alpha_W$ and probability weights $p_W$, define
\begin{equation}\label{eq:profile-def}
 \mathfrak P_K(\mathbf D,\boldsymbol\alpha;\mathbf p)
 =
 \frac{\left(\sum_Wp_WD_W\right)^K}
 {\left(\sum_Wp_WD_W^{K/(K-1)}
          \alpha_W^{-1/(K-1)}\right)^{K-1}}.
\end{equation}

\subsection{The four-dimensional profile theorem}

Let
\[
 \cV=\disjunion_{W\in\cW}\cV_W
\]
be a disjoint union by indices.  Repetitions are not identified.  Given the input
shading, define its initial positive-mass parent support
\begin{equation}\label{eq:initial-parent-support}
 \cW_0=\left\{W\in\cW:
 M_W^{(0)}:=\sum_{V\in\cV_W}\abs{Y(V)}>0\right\}.
\end{equation}
When the global shaded mass is positive, $\cW_0$ is nonempty.  All profile sums and
the frozen coarse denominator below are over $\cW_0$.  Removing the other parent
groups preserves shaded mass and can only increase the fine density.  We denote the
indexed buffer family $(W^\sharp)_{W\in\cW_0}$ by $\cW_0^\sharp$.

\begin{theorem}[Profile-stable buffered factoring in $\R^4$]
\label{thm:main}
Fix $\Lambda_W\ge1$ and
$0<w_1\le w_2\le w_3\le w_4$.  Suppose that:
\begin{enumerate}[label=\textup{(H\arabic*)}]
\item $V\subset W$ whenever $V\in\cV_W$;
\item every parent satisfies
\[
 \Lambda_W^{-1}w_j\le a_j(W)\le\Lambda_Ww_j,
\qquad 1\le j\le4;
\]
\item the $Y(V)\subset V$ are measurable and
\[
 \lambda=\frac{\sum_{V\in\cV}\abs{Y(V)}}{\sum_{V\in\cV}\abs V}>0.
\]
\end{enumerate}
There is no comparability hypothesis on child axes, volumes, aspect ratios, or
orientations.  Set $h=w_1/100$, $W^\sharp=N_{w_1}(W)$, and
\[
 \Lfac=2+\log(2+\#\cV)+\log(2+\#\cW)+\log(2+w_4/w_1).
\]
For every $\eps>0$ there are $K_\eps\ge2$ and $A_\eps<\infty$ such that the
following holds.  For $W\in\cW_0$, form the frozen data above at radius $h$, put
\[
 p_W=\frac{\abs W}{\sum_{W'\in\cW_0}\abs{W'}},
 \qquad q_\eps=\frac{K_\eps}{K_\eps-1},
\]
and define
\begin{equation}\label{eq:intrinsic-main-profile}
 \mathfrak P_\eps^{\rm int}
 =
 \frac{\left(\sum_{W\in\cW_0}p_WD_W\right)^{K_\eps}}
 {\left(\sum_{W\in\cW_0}p_WD_W^{q_\eps}
       (\kappa_W^{(K_\eps)})^{-1/(K_\eps-1)}
  \right)^{K_\eps-1}}.
\end{equation}
Then there exist one refinement $Y'$, coarse shadings
$Z(W^\sharp)\subset W^\sharp$ for every $W\in\cW_0$, a nonempty final
positive-mass set
\[
 \cW_+=\left\{W\in\cW_0:
   M'_W:=\sum_{V\in\cV_W}\abs{Y'(V)}>0\right\},
\]
positive dyadic $a,b,c$, and $\sigma>0$, with $Z(W^\sharp)=\varnothing$ for
$W\notin\cW_+$, such that
\begin{align}
 \sum_V\abs{Y'(V)}
 &\gtrsim \Lfac^{-A_\eps}\sum_V\abs{Y(V)},
 \label{eq:main-mass}\\
 U'_W\subset Z(W^\sharp)
 &\subset N_{w_1}(U'_W)\cap W^\sharp
 \quad(W\in\cW_+),
 \label{eq:main-compat}\\
 \lambda(\cW_0^\sharp,Z)
 &\gtrsim_{\eps,\Lambda_W}
 \Lfac^{-A_\eps}\left(\frac{w_1}{w_4}\right)^\eps
 \mathfrak P_\eps^{\rm int}\lambda^{K_\eps},
 \label{eq:main-coarse-density}\\
 a\le m_{\cV_W,Y'}(x)&<2a
 \quad(x\in U'_W,\ W\in\cW_+),
 \label{eq:main-inner}\\
 c\le m_{\cV,Y'}(x)&<2c
 \quad(x\in U(\cV,Y')),
 \label{eq:main-global}\\
 b\le m_{\cW_0^\sharp,Z}(x)&<2b
 \quad(x\in U(\cW_0^\sharp,Z)).
 \label{eq:main-coarse-mult}
\end{align}
Here $U'_W=\bigcup_{V\in\cV_W}Y'(V)$.  For every $W\in\cW_+$,
\begin{equation}\label{eq:main-product}
 \mu(\cV,Y)\lesssim_{\eps,\Lambda_W}\Lfac^{A_\eps}
 \mu(\cW_0^\sharp,Z)\mu(\cV_W,Y').
\end{equation}
For every $x\in\R^4$,
\begin{equation}\label{eq:main-fibre}
 (\cV)_{Y'}(x)=
 \disjunion_{\substack{W\in\cW_0\\x\in Z(W^\sharp)}}
 (\cV_W)_{Y'}(x).
\end{equation}
Moreover,
\begin{align}
 \abs{U(\cV,Y')\cap B(x,w_1)}
 &\asymp\sigma
 \quad(x\in U(\cW_0^\sharp,Z)),
 \label{eq:main-local-uniformity}\\
 \Delta_{\max}(\cW_0^\sharp)
 &\lesssim_{\Lambda_W}\Delta_{\max}(\cW).
 \label{eq:main-kt}
\end{align}
\end{theorem}

\paragraph{Coarse normalization.}
The set $\cW_0$ is fixed by the input shading, whereas $\cW_+$ records the parents
that retain positive mass after the refinement.  The density denominator remains
$\sum_{W\in\cW_0}\abs{W^\sharp}$; parents in
$\cW_0\setminus\cW_+$ carry empty coarse shading but stay in that denominator.
Moreover, $Z(W^\sharp)$ is a union of selected cells, so a point may belong to a
coarse shading even when its fine fibre for that parent is empty.  The fibre identity
is formulated accordingly.

\begin{corollary}[Elementary $C$-Frostman profile]\label{cor:main-elementary}
In Theorem~\ref{thm:main}, suppose additionally that every
$\cV_W$, $W\in\cW_0$, is $C$-Frostman in $W$ for one $C\ge1$.  Define
\begin{equation}\label{eq:elementary-main-profile}
 \mathfrak P_\eps^{\rm elem}
 =
 \frac{\left(\sum_{W\in\cW_0}p_WD_W\right)^{K_\eps}}
 {\left(\sum_{W\in\cW_0}p_WD_W^{q_\eps}
       (\eta_W^{(K_\eps)})^{-1/(K_\eps-1)}
 \right)^{K_\eps-1}}.
\end{equation}
The same objects may be chosen so that
\begin{equation}\label{eq:elementary-main-density}
 \lambda(\cW_0^\sharp,Z)
 \gtrsim_{\eps,\Lambda_W}
 \Lfac^{-A_\eps}\left(\frac{w_1}{w_4}\right)^\eps
 C^{-1}\mathfrak P_\eps^{\rm elem}\lambda^{K_\eps}.
\end{equation}
\end{corollary}

The indexed formulation permits arbitrary repetitions.  Their contributions remain
visible in the cardinalities entering $\Lfac$, in the multiplicities, and in the
displayed profile data, while the implicit constants require no separate bound on
duplicate multiplicity.  Likewise, the child minimum scales and child axis ratios
enter through the explicit inflation and profile quantities rather than through
$\Lfac$ or an implicit constant.

\section{A codimension-one density-response lift}\label{sec:lift}

The local four-dimensional estimate will be obtained from a three-dimensional union
bound after projection.  The following measure-theoretic lemma returns the projected
response to the original dimension.  It uses only layer cake and convexity of the
response function, so the density exponent is preserved without introducing a new
geometric loss.

\begin{theorem}[Density-response lift]\label{thm:lift}
Let $I\subset\R$ be an interval of length $H>0$, and let
$\cQ=(Q_i)_{i\in J}$ be a nonempty finite indexed family of measurable subsets of
$\R^d$ with
\[
 D:=\sum_{i\in J}\abs{Q_i}\in(0,\infty).
\]
Suppose that for every measurable $E_i\subset Q_i$ one has
\begin{equation}\label{eq:base-cu}
 \abs{\bigcup_iE_i}\ge A(\cQ)\,\Psi\left(
          \frac{\sum_i\abs{E_i}}{D}\right),
\end{equation}
where $A(\cQ)\ge0$ is finite and independent of the shadings,
$\Psi:[0,1]\to[0,\infty)$ is convex, and $\Psi(0)=0$.  Then, for every measurable
$S_i\subset I\times Q_i$,
\begin{equation}\label{eq:lift-conclusion}
 \abs{\bigcup_iS_i}\ge H A(\cQ)\,\Psi\left(
        \frac{\sum_i\abs{S_i}}{HD}\right).
\end{equation}
\end{theorem}

\begin{proof}
For almost every $z\in\R^d$, define the one-dimensional fibre and its length by
\[
 (S_i)_z=\{s\in I:(s,z)\in S_i\},
 \qquad f_i(z)=\abs{(S_i)_z}.
\]
Tonelli's theorem gives measurable representatives with $0\le f_i\le H$, with
$f_i=0$ almost everywhere off $Q_i$, and
\begin{equation}\label{eq:fibre-integral}
 \int_{\R^d}f_i(z)\,dz=\abs{S_i}.
\end{equation}
The fibre of the finite union contains each individual fibre, so
\[
 \abs{(\bigcup_iS_i)_z}\ge\max_i f_i(z)
\]
for almost every $z$.  Put
\[
 E_i(t)=\{z:f_i(z)>t\},\qquad
 \rho(t)=\frac{\sum_i\abs{E_i(t)}}D,\qquad 0<t<H.
\]
The sets $E_i(t)$ are measurable subsets of $Q_i$, and layer cake gives
\begin{align}
 \abs{\bigcup_iS_i}
 &\ge\int_{\R^d}\max_i f_i(z)\,dz \\
 &=\int_0^H\abs{\bigcup_iE_i(t)}\,dt
 \ge A(\cQ)\int_0^H\Psi(\rho(t))\,dt.
 \label{eq:lift-layer}
\end{align}
Normalized Jensen on $[0,H]$ and a second layer-cake computation yield
\begin{align*}
 \int_0^H\Psi(\rho(t))\,dt
 &\ge H\Psi\left(\frac1H\int_0^H\rho(t)\,dt\right),\\
 \int_0^H\rho(t)\,dt
 &=\frac1D\sum_i\int_{\R^d}f_i(z)\,dz
 =\frac1D\sum_i\abs{S_i}.
\end{align*}
Substitution in \eqref{eq:lift-layer} proves \eqref{eq:lift-conclusion}.  The
normalization uses the fixed body mass $HD$ rather than the shaded mass.  Consequently
the same argument includes the zero-density case, where $\Psi(0)=0$.
\end{proof}

Repeated indices pass through the lift unchanged: the union and fibre maximum are
geometric, whereas $D$ and the density numerator retain the indexed sums.  The
lower-dimensional estimate must therefore already be valid for the indexed family.
Section~\ref{sec:indexed} establishes precisely that input before the lift is applied.

\begin{proposition}[Product sharpness]\label{prop:lift-product}
Fix measurable $E_i\subset Q_i$ and take $S_i=I\times E_i$.  Then every inequality
in the fibre and Jensen part of the proof of Theorem~\ref{thm:lift} is an equality.
Consequently, neither the factor $H$ nor the response at the mean density can be
improved uniformly by the lifting argument alone.
\end{proposition}

\begin{proof}
For almost every $0<t<H$, one has $f_i=H\ind_{E_i}$ and
$E_i(t)=E_i$.  Therefore
\[
 \abs{\bigcup_iS_i}=H\abs{\bigcup_iE_i},
 \qquad \rho(t)=\frac{\sum_i\abs{E_i}}D.
\]
The fibre comparison and Jensen are both equalities.  Any uniform improvement in
\eqref{eq:lift-conclusion} would therefore give the same improvement in
\eqref{eq:base-cu}.
\end{proof}

\begin{remark}[Nonconvex responses]
Without convexity, \eqref{eq:lift-layer} remains valid, but the value of $\Psi$ at
the mean need not be controlled.  If $\co(\Psi)$ denotes the lower convex envelope,
the automatic conclusion is
\[
 \abs{\bigcup_iS_i}\ge H A(\cQ)\,\co(\Psi)\left(
        \frac{\sum_i\abs{S_i}}{HD}\right).
\]
In the application below, $\Psi(t)=t^{K_\eps}$ with $K_\eps\ge1$, so the response
is convex and its exponent is preserved exactly.
\end{remark}

\section{Indexed convex unions in three dimensions}\label{sec:indexed}

Projection creates the first indexing difficulty in the local argument.  The
three-dimensional convex-union theorem is formulated for essentially-distinct
bodies, while two different four-dimensional child indices may have the same
projection.  Passing to a geometric set of representatives would discard precisely
the multiplicities that the final factorization must retain.  We instead derive the
localized essentially-distinct estimate needed here, organize coincident projections
by affine-adapted canonical boxes, and recover the indexed normalization through a
collision-layer cancellation.

\subsection{The essentially-distinct input}

Two measurable sets $Q,Q'\subset\R^3$ are \emph{essentially distinct} if
\begin{equation}\label{eq:essentially-distinct}
 \abs{Q\cap Q'}\le\frac12\max(\abs Q,\abs{Q'}).
\end{equation}
This is the convention of Wang and Zahl
\cite{wangzahl2025convexunions}.  We begin with their set-valued hypothesis; the
indexed extension is proved later in this section.

\begin{proposition}[Localized essentially-distinct union estimate]
\label{prop:ed-cu3}
Let $0<q_1\le q_2\le q_3$, and let $\cQ$ be a nonempty essentially-distinct family
of convex bodies contained in a ball of radius $O(q_3)$.  Suppose their ordered John
semiaxes are comparable, with fixed constants, to $q_1,q_2,q_3$.  For measurable
$E(Q)\subset Q$, put
\[
 \rho=\frac{\sum_{Q\in\cQ}\abs{E(Q)}}{\sum_{Q\in\cQ}\abs Q},
 \qquad m=\Delta_{\max}(\cQ).
\]
For every $\eps>0$ there is $K_\eps\ge1$ such that
\begin{equation}\label{eq:ed-cu3}
 \abs{\bigcup_{Q\in\cQ}E(Q)}
 \gtrsim_\eps
 \left(\frac{q_1}{q_3}\right)^\eps
 \rho^{K_\eps}m^{-1}\sum_{Q\in\cQ}\abs Q.
\end{equation}
The constants depend on the fixed axis-comparability constants and on the fixed
constant hidden in the ambient radius $O(q_3)$.
\end{proposition}

\begin{proof}
It suffices to prove the claim for $0<\eps\le1$, since for
$0<a\le1$ the bound with exponent $\min\{\eps,1\}$ implies the bound with exponent
$\eps$. After an affine scaling of all lengths by $q_3^{-1}$, write
$a=q_1/q_3$, $b=q_2/q_3$, and let $N=\#\cQ$ and
$v\asymp ab$ be the common volume.  We use the positive-parameter assertion
$F(t,t)$ for $0<t\le2/3$.  In the source notation, it is obtained as follows.
Their Theorem~1.9 gives $D(0,0)$; at $\sigma=\omega=t$, Definition~1.5 makes
$D(t,t)$ weaker by the factor $(\#\mathcal T)^{-t}\le1$.
Proposition~1.6 gives $E(t,t)$, and Proposition~5.14 gives $F(t,t)$
\cite{wangzahl2025convexunions}. In the source notation it gives, for families of
density at least $a^{\eta}$, where
$\eta:=\eta_{\mathrm{WZ}}(\eps/2;t,t)>0$,
\begin{equation}\label{eq:wz-positive}
 \abs{\bigcup_QE(Q)}\gtrsim
 a^{\eps/2}b^t m^{-1}Nv
 \bigl(m^{-3/2}\ell Nv^{1/2}\bigr)^{-t}
 \mathcal D(\cQ)^{-t},
\end{equation}
where $\ell$ is the source's Frostman slab constant and $\mathcal D$ is its
geometric concentration quantity.

The localization gives $Nv\lesssim m$, while testing $m$ on a member of the family
gives $m\gtrsim1$.  A slab containing an $a\times b\times1$ John-normalized body
has thickness and volume $\gtrsim a$, and hence $\ell\lesssim a^{-1}$.
For every $Q$, the localized subfamily $\cQ[N_b(Q)]$ has Katz--Tao constant at
most $m$ by subfamily monotonicity.  Remark~5.6, Situation~2, therefore gives
$\mathcal D(\cQ)\lesssim m^{1/2}$.  Consequently, with
$B:=m^{-3/2}\ell Nv^{1/2}$,
\[
 B\mathcal D(\cQ)
 \lesssim m^{-1}\ell Nv^{1/2}
 \lesssim \ell v^{-1/2}
 \lesssim a^{-2}.
\]
Choose $t=\eps/10$.  Since $b\ge a$,
\[
 b^tB^{-t}\mathcal D(\cQ)^{-t}
 =b^t\bigl(B\mathcal D(\cQ)\bigr)^{-t}
 \gtrsim b^ta^{2t}\ge a^{3t}.
\]
For $0<a\le1$, the exponent $\eps/2+3\eps/10=4\eps/5$ is smaller than $\eps$;
thus \eqref{eq:wz-positive} is at least the right side of
\eqref{eq:ed-cu3} whenever $\rho\ge a^{\eta}$, because
$\rho^{K}\le1$.

It remains to remove the source density threshold.  If
$0<\rho<a^{\eta}$, one shading has measure $\gtrsim\rho v$, and
$v\asymp ab\ge a^2$.  Since $m^{-1}Nv\lesssim1$, any
\[
 K_\eps\ge1+\frac{2}{\eta}
\]
satisfies
\[
 \abs{\bigcup_QE(Q)}\gtrsim\rho a^2
 \gtrsim a^\eps\rho^{K_\eps}m^{-1}Nv.
\]
The zero-density case is immediate, completing the all-density localized estimate.
The source's ``outer John ellipsoid'' is used with the fixed-dilate convention of
Section~\ref{sec:prelim}, which changes only the fixed comparability constants.
\end{proof}

Thus Proposition~\ref{prop:ed-cu3} is the localized all-density consequence needed
below.  Its reduction keeps the positive source parameters explicit and supplies the
remaining small-density range by the largest-shading argument.  Appendix
\ref{app:external} records the full notation dictionary for this derivation.

\subsection{Affine-adapted canonical boxes}

A Euclidean net in all Euler angles does not give aspect-ratio-uniform collision
control.  For example,
boxes of shape $\delta\times\delta\times1$ with a common center and long axis, rotated
in roll increments of size $\delta$, form a clique of size comparable to
$\delta^{-1}$ in the non-essential-distinctness graph.  The parameter metric must
therefore measure the relative affine transform at the scale of the box itself.

We first isolate the parameter-space fact used in the discretization. Put
$C=[-1,1]^3$ and, for a fixed positive diagonal matrix $D$, write
\[
 B(c,U)=c+UDC,\qquad c\in\R^3,\quad U\in SO(3).
\]
Let
\[
 H_D=\{S\in SO(3):SDC=DC\}.
\]
This is the orientation-preserving signed-permutation symmetry group of the
rectangular box, and $\#H_D\le24$. For two frame representations define
\[
 A=D^{-1}U^TU'D,\qquad b=D^{-1}U^T(c'-c),
\]
Here $\abs{\cdot}$ is the Euclidean norm on vectors and $\norm{\cdot}$ is its induced
operator norm on matrices.  Both norms are fixed before $D$ is chosen, so all
dependence on the affine shape appears explicitly through the conjugation by $D$.
Set
\begin{equation}\label{eq:parameter-gauge}
 d_0((c,U),(c',U'))=
 \max\{\norm{A-I},\norm{A^{-1}-I},\abs b,\abs{A^{-1}b}\}.
\end{equation}
Minimize \eqref{eq:parameter-gauge} over the left and right actions of $H_D$ and
denote the resulting symmetric gauge on geometric boxes by $d_D$.

\begin{lemma}[Local affine parameter packing]\label{lem:parameter-packing}
There are dimensional constants $\eta_*>0$ and $C<\infty$, independent of the
condition number of $D$, with the following properties.
\begin{enumerate}[label=\textup{(\roman*)}]
\item The gauge $d_D$ is independent of the frame representations of the two
geometric boxes. If $d_D(B,B')\le\eta\le\eta_*$, then, after choosing representatives
realizing the minimum,
\[
 B'\subset(1+C\eta)B,
 \qquad B\subset(1+C\eta)B',
\]
where the dilations are about the respective centers.
\item If $LB$ denotes dilation about the center by $L\ge1$, then
\[
 d_{LD}(LB,LB')\ge L^{-1}d_D(B,B').
\]
\item Fix $B_0$ of shape $D$ and $0<\zeta\le1$. If boxes $B_\beta$ of the same shape satisfy
$d_D(B_\beta,B_\gamma)>\zeta$ for $\beta\ne\gamma$ and
\[
 \abs{B_0\cap B_\beta}>\tfrac12\abs{B_0},
\]
then their number is $O(\zeta^{-12})$, with a dimensional implicit constant.
\end{enumerate}
\end{lemma}

\begin{proof}
All frame representations of one chosen geometric box differ by right multiplication
by an element of $H_D$. Such an element is a signed permutation among equal diagonal
entries of $D$, so it commutes with $D$. Consequently a change of the two
representatives sends $(b,A)$ to one of the declared finite left/right transforms.
The minimum is therefore a function of the geometric boxes. Swapping the boxes sends
$(b,A)$ to $(-A^{-1}b,A^{-1})$, so the four terms in
\eqref{eq:parameter-gauge} also prove symmetry.

In the coordinates of the first box, the second is $b+AC$. If the gauge is at most
$\eta$, choose minimizing representatives. For $x\in C$,
\[
 b+Ax=x+\bigl(b+(A-I)x\bigr)\in(1+C\eta)C.
\]
The two inverse terms give the reverse containment, establishing (i). Under common
dilation, $A$ is unchanged and $b$ is replaced by $b/L$. For each fixed pair of
symmetry representatives, the maximum in \eqref{eq:parameter-gauge} decreases by at
most the factor $L$; taking the finite minimum proves (ii).

For (iii), normalize $B_0$ affinely to $C$. A neighbor has the form $b+AC$ with
$\det A=1$. The set
\[
 \Omega=\{y\in C:b+Ay\in C\}
\]
has measure greater than $\abs C/2$. If $a_r$ is a row of $A$, then
$\Omega$ lies in the slab $\{y\in C:\abs{b_r+a_r\cdot y}\le1\}$. Slicing in a
coordinate on which $a_r$ has largest absolute component shows that the volume of
this slab inside $C$ is $O(1/\norm{a_r}_\infty)$. Thus every row of $A$ is bounded.
Nonempty intersection gives $b=x-Ay$ for some $x,y\in C$, and hence bounds $b$.
Since $\det A=1$, the cofactor formula bounds $A^{-1}$ as well. All neighboring
parameters $(b,A)$ therefore lie in one fixed compact subset of $\R^{12}$.

Choose for each neighbor a representative in this central chart. For two such
parameters,
\[
 A_{\beta\gamma}=A_\beta^{-1}A_\gamma,
 \qquad b_{\beta\gamma}=A_\beta^{-1}(b_\gamma-b_\beta).
\]
On the fixed compact set, the identity-representative value of $d_0$ is at most a
dimensional constant times
$\abs{b_\beta-b_\gamma}+\norm{A_\beta-A_\gamma}$. Since $d_D$ is the minimum over
representatives, $d_D(B_\beta,B_\gamma)>\zeta$ forces ordinary Euclidean separation
$\gtrsim\zeta$ of these parameters. Euclidean packing in the compact subset of
$\R^{12}$ now gives $O(\zeta^{-12})$ neighbors.
\end{proof}

\begin{corollary}[Condition-number-independent collision degree]
\label{cor:condition-independent-degree}
Let $D$ be any positive diagonal matrix and let $0<\zeta\le1$. If a family of
shape-$D$ boxes is $\zeta$-separated in $d_D$, then the graph joining two boxes whose
intersection has more than half the volume of either box has maximum degree
$O(\zeta^{-12})$. The implicit constant is dimensional and is independent of
$\norm D\norm{D^{-1}}$.
\end{corollary}

\begin{proof}
Fix a vertex $B_0$. All boxes have the same volume, so every neighbor satisfies the
intersection hypothesis in Lemma~\ref{lem:parameter-packing}(iii). That lemma counts
the neighbors of $B_0$ by $O(\zeta^{-12})$, uniformly in $D$.
\end{proof}

\begin{lemma}[Canonical boxes]\label{lem:canonical-boxes}
Fix $C_0\ge1$.  Let $(Q_i)_{i\in I}$ be a finite indexed family of convex bodies in
$\R^3$ whose ordered John semiaxes lie in $[q_j,C_0q_j]$, $1\le j\le3$.  There are
an assignment $i\mapsto\alpha(i)$ and boxes $R_\alpha$ such that:
\begin{enumerate}[label=\textup{(\roman*)}]
\item $Q_i\subset R_{\alpha(i)}$ and
$\abs{Q_i}\asymp_{C_0}\abs{R_{\alpha(i)}}$;
\item the distinct geometric boxes $R_\alpha$ can be colored with $O_{C_0}(1)$
colors so that boxes of one color are essentially distinct;
\item the number of colors is independent of $q_3/q_1$.
\end{enumerate}
\end{lemma}

\begin{proof}
Fix a small dimensional $\eta_0>0$ and round each John semiaxis upward to the least
number of the form $q_j(1+\eta_0)^k$.  There are at most
\begin{equation}\label{eq:axis-rounding-count}
 \left(1+\left\lceil\log_{1+\eta_0}C_0\right\rceil\right)^3
\end{equation}
rounded triples. We use disjoint color palettes for different triples and work with
one exact triple. Choose an arbitrary John eigenframe for each member; when eigenvalues
repeat this is merely a tie-break, and all estimates below are uniform in that choice.
After incorporating a fixed John dilation, every member lies in a box
\[
 B(c,U)=c+UD[-1,1]^3,\qquad D=\operatorname{diag}(q_1,q_2,q_3),
\]
of comparable volume.

The group $H_D$ in Lemma~\ref{lem:parameter-packing} identifies only different frame
representations of a resulting geometric box; it is not being used to remove the
continuous eigenframe tie-break. Choose greedily a maximal $\eta$-separated set of
distinct input boxes in the symmetric gauge $d_D$,
where $\eta>0$ is a sufficiently small dimensional constant.  Maximality assigns
every input box $B_i$ to a selected box $B^0_\alpha$ at distance at most $\eta$.
Lemma~\ref{lem:parameter-packing}(i) gives
\[
 B_i\subset L B^0_\alpha,\qquad L=1+C\eta.
\]
Set $R_\alpha=LB^0_\alpha$.  The containment and volume comparison follow.  Common
dilation leaves the selected boxes at least $L^{-1}\eta$-separated by
Lemma~\ref{lem:parameter-packing}(ii). If $R_\beta$ is not essentially distinct from
$R_\alpha$, their equal volumes give
$\abs{R_\alpha\cap R_\beta}>\abs{R_\alpha}/2$.
Corollary~\ref{cor:condition-independent-degree} bounds
the degree of this collision graph independently of the aspect ratio. A bounded
greedy coloring, followed by the finite number of palettes in
\eqref{eq:axis-rounding-count}, completes the proof.
\end{proof}

\subsection{Collision cancellation}

We now return to the original indices.  Canonical boxes merge geometrically close
projections, but the size of a collision cluster is retained as indexed mass.  On a
layer where the cluster sizes are comparable, the same cluster size improves the
Katz--Tao density of the geometric representatives.  These two effects cancel in the
union estimate.

For an indexed family $\cQ=(Q_i)_{i\in I}$ with comparable John dimensions, choose
canonical boxes as in Lemma~\ref{lem:canonical-boxes}.  Let
\[
 I_\alpha=\{i:\alpha(i)=\alpha\},\quad
 n_\alpha=\#I_\alpha,\quad
 F_\alpha=\bigcup_{i\in I_\alpha}E_i,
\]
and put
\[
 S=\sum_i\abs{Q_i},\qquad M=\sum_i\abs{E_i},
 \qquad\rho=M/S,\qquad m=\Delta_{\max}(\cQ).
\]

\begin{lemma}[Density gain on a collision layer]\label{lem:collision-density}
If $n\le n_\alpha<2n$ on a set $\cA_n$ of canonical boxes, then
\begin{equation}\label{eq:collision-kt}
 \Delta_{\max}(\{R_\alpha:\alpha\in\cA_n\})\lesssim m/n.
\end{equation}
\end{lemma}

\begin{proof}
For every convex body $K$, containment and two-sided volume comparison give
\begin{align*}
 \sum_{\substack{\alpha\in\cA_n\\R_\alpha\subset K}}\abs{R_\alpha}
 &\lesssim\frac1n
 \sum_{\substack{\alpha\in\cA_n\\R_\alpha\subset K}}
 \sum_{i\in I_\alpha}\abs{Q_i}\\
 &\le\frac1n\sum_{i:Q_i\subset K}\abs{Q_i}
 \le\frac mn\abs K.
\end{align*}
The middle implication uses only
$R_\alpha\subset K\Rightarrow Q_i\subset R_\alpha\subset K$.
\end{proof}

\begin{theorem}[Essentially-distinct to indexed]\label{thm:icu3}
Under the hypotheses of Proposition~\ref{prop:ed-cu3}, but allowing arbitrary indexed
repetitions, let $N=\#I$ and
$L_N=1+\lceil\log_2N\rceil$.  Then
\begin{equation}\label{eq:icu3}
 \abs{\bigcup_iE_i}\gtrsim_\eps
 \left(\frac{q_1}{q_3}\right)^\eps
 L_N^{-1}\rho^{K_\eps}m^{-1}\sum_i\abs{Q_i}.
\end{equation}
The estimate remains true without a common ambient ball.
\end{theorem}

\begin{proof}
First suppose that the family is localized.  On a collision layer
$\cA_n=\{\alpha:n\le n_\alpha<2n\}$, write $S_n,M_n$ for the body and shading
masses of the corresponding indices and $\rho_n=M_n/S_n$.  The elementary estimate
\[
 \abs{F_\alpha}\ge n_\alpha^{-1}
       \sum_{i\in I_\alpha}\abs{E_i}
\]
and volume comparison imply
\begin{equation}\label{eq:collision-masses}
 \sum_{\alpha\in\cA_n}\abs{R_\alpha}\asymp S_n/n,
 \qquad
 \sum_{\alpha\in\cA_n}\abs{F_\alpha}\gtrsim M_n/n.
\end{equation}
Split the canonical boxes into the bounded number of colors supplied by
Lemma~\ref{lem:canonical-boxes}. For a color $\nu$, let $S_{n,\nu}$ and
$M_{n,\nu}$ be the body and shading masses of its assigned input indices, and put
$\rho_{n,\nu}=M_{n,\nu}/S_{n,\nu}$ when $S_{n,\nu}>0$. Within that color,
Proposition~\ref{prop:ed-cu3}, Lemma~\ref{lem:collision-density}, and
\eqref{eq:collision-masses} give the lower bound
\[
 \left(\frac{q_1}{q_3}\right)^\eps
 m^{-1}S_{n,\nu}\rho_{n,\nu}^{K_\eps}
\]
for the union of its merged shadings. The color unions need not be disjoint, so we do
not add these bounds. Instead, weighted Jensen gives
\[
 \sum_\nu S_{n,\nu}\rho_{n,\nu}^{K_\eps}
 \ge S_n\rho_n^{K_\eps}.
\]
Because the number of colors is bounded, one color $\nu_*$ contributes a fixed
fraction of the right-hand side. The union over the entire collision layer contains
the union for $\nu_*$, and therefore
\begin{equation}\label{eq:one-collision-layer}
 \abs{\bigcup_{\alpha\in\cA_n}F_\alpha}
 \gtrsim_\eps
 \left(\frac{q_1}{q_3}\right)^\eps
 \rho_n^{K_\eps}m^{-1}S_n.
\end{equation}
At the level of powers, the factor $1/n$ in canonical-box mass is canceled by the
factor $n$ in $(m/n)^{-1}$.

There are at most $L_N$ nonempty layers.  Since $K_\eps\ge1$, weighted Jensen gives
\begin{equation}\label{eq:layer-jensen}
 \sum_n S_n\rho_n^{K_\eps}
 \ge S\left(\sum_n\frac{S_n}{S}\rho_n\right)^{K_\eps}
 =S\rho^{K_\eps}.
\end{equation}
Hence one layer contributes at least $L_N^{-1}S\rho^{K_\eps}$, and
\eqref{eq:icu3} follows from \eqref{eq:one-collision-layer}.

For a nonlocalized family, assign each body by its John center to a cube of a
$q_3$-lattice.  A fixed enlargement of that cube contains the whole body.  Color the
lattice with a dimensional number of colors so that equal-color enlargements are
disjoint.  Apply the localized argument in each cube.  Within a fixed lattice color,
the resulting union lower bounds add because the enlargements are disjoint; weighted
Jensen combines their body-mass densities. Across lattice colors the unions can
overlap, so weighted Jensen followed by a bounded-color pigeonhole selects one color;
the full union contains that selected color's union. This costs only a dimensional
constant. The global bound $m$ dominates each local Katz--Tao density, which gives the
final assertion.
\end{proof}

Complete duplication illustrates the normalization.  If all $Q_i=Q$ and all
$E_i=E$, then
$m=N$, and the right side of \eqref{eq:icu3} is bounded by a constant multiple of
$\rho^{K_\eps}\abs Q\le\abs E$.  If the shadings are disjoint, the cluster union
bound is conservative but valid.  Throughout the construction the original indices
remain in the mass and Katz--Tao quantities.

\subsection{The local Frostman form}

The preceding theorem assumes comparable John dimensions.  For the projected family
arising in Section~\ref{sec:local-r4}, each body instead contains an $h$-ball and lies
in a common box of longest side $R$.  Dyadic decomposition of the three John axes
reduces this more flexible setting to Theorem~\ref{thm:icu3}; relative Frostman
control supplies the uniform Katz--Tao bound in every bin.

\begin{theorem}[Indexed local Frostman union estimate]\label{thm:lf3}
Let $B'\subset\R^3$ be a box of longest side $R$, and let
$\cQ=(Q_i)_{i\in I}$ be a nonempty finite indexed family of convex bodies contained
in $B'$.  Suppose every
$Q_i$ contains a ball of radius $h$, has diameter $O(R)$, and
$C_F(\cQ,B')\le C_F$.  For measurable $E_i\subset Q_i$, let
\[
 \rho=\frac{\sum_i\abs{E_i}}{\sum_i\abs{Q_i}},\qquad N=\#I,
\]
and set
\[
 L_N=1+\lceil\log_2N\rceil,\qquad
 J_h=\bigl(1+\lceil\log_+(R/h)\rceil\bigr)^3.
\]
Then, for every $\eps>0$,
\begin{equation}\label{eq:lf3}
 \abs{\bigcup_iE_i}\gtrsim_\eps
 \left(\frac hR\right)^\eps
 L_N^{-1}J_h^{-1}C_F^{-1}\rho^{K_\eps}\abs{B'}.
\end{equation}
\end{theorem}

\begin{proof}
Put $S=\sum_i\abs{Q_i}$ and $D=S/\abs{B'}$.  The relative Frostman hypothesis
implies the absolute indexed bound
\begin{equation}\label{eq:absolute-from-relative}
 \Delta_{\max}(\cQ_0)\le C_FD
\end{equation}
for every subfamily $\cQ_0$.  Indeed, for a convex test body $K$, the bodies counted
in $K$ are also counted in $K\cap B'$, while
$\abs K\ge\abs{K\cap B'}$; use \eqref{eq:frostman-def} and subfamily monotonicity.

Because $Q_i$ contains an $h$-ball and lies in $B'$, its shortest John semiaxis is
$\gtrsim h$ and its longest semiaxis is $O(R)$.  Divide the three ordered semiaxes
into dyadic triples.  There are at most $O(J_h)$ nonempty bins.  For bin $j$, let
$S_j,M_j$ be its body and shading masses and $\rho_j=M_j/S_j$.  Theorem~\ref{thm:icu3}
and \eqref{eq:absolute-from-relative} give
\[
 \abs{\bigcup_{i\text{ in bin }j}E_i}
 \gtrsim_\eps
 \left(\frac hR\right)^\eps L_N^{-1}
 (C_FD)^{-1}S_j\rho_j^{K_\eps}
 =\left(\frac hR\right)^\eps L_N^{-1}C_F^{-1}\abs{B'}
   \frac{S_j}{S}\rho_j^{K_\eps}.
\]
Finally,
\[
 \sum_j\frac{S_j}{S}\rho_j^{K_\eps}
 \ge\left(\sum_j\frac{S_j}{S}\rho_j\right)^{K_\eps}
 =\rho^{K_\eps}.
\]
One of the at most $O(J_h)$ bins therefore gives \eqref{eq:lf3}.  When $h\gtrsim R$
there is one bin, and the same proof applies.  Empty shadings give zero right side.
\end{proof}

The denominator in Theorem~\ref{thm:lf3} is the body mass of the fixed indexed
family, including indices whose current shading is empty.  Consequently the same
estimate can be applied to every measurable refinement produced later by the
cellular construction.

\section{The local four-dimensional transfer}\label{sec:local-r4}

We now convert the indexed three-dimensional estimate into a parentwise statement in
four dimensions.  The main issue is that thickening children of different geometry
produces different volume weights.  We first quantify thickening for arbitrary
measurable shadings, then project along a shortest parent direction so that the
volume of each thickened body is comparable to $h$ times its projected volume.  The resulting weighted
Frostman data can be passed to Theorem~\ref{thm:lf3} and lifted back by
Theorem~\ref{thm:lift}.

\subsection{Convex thickening}

The cellular argument will apply the local estimate after several measurable
refinements.  We therefore need a thickening bound that depends on the density of an
arbitrary measurable shading, rather than on any regularity of that shading.

\begin{lemma}[Thickening an arbitrary measurable shading]\label{lem:thickening}
For every convex body $V\subset\R^n$, every measurable $Y\subset V$, and every
$h>0$,
\begin{equation}\label{eq:thickening}
 \abs{N_h(Y)}\ge c_n\frac{\abs Y}{\abs V}\abs{N_h(V)}.
\end{equation}
\end{lemma}

\begin{proof}
Use John coordinates with semiaxes $a_1,\ldots,a_n$.  From
\eqref{eq:john-inclusions} and a containing rectangular box,
\begin{equation}\label{eq:nhv-upper}
 \abs V\asymp_n\prod_j a_j,\qquad
 \abs{N_h(V)}\lesssim_n\prod_j(a_j+h).
\end{equation}
Let $L$ be the span of the axes with $a_j\ge h$ and $S=L^\perp$.  Write $m=\dim S$
and let $Y_x$ be the $S$-fibre over $x\in L$.  The John box bounds every fibre by
$C_n\prod_{a_j<h}a_j$.  If $X=\{x:\abs{Y_x}>0\}$, Fubini gives
\begin{equation}\label{eq:projection-support}
 \abs X\gtrsim_n\frac{\abs Y}{\prod_{a_j<h}a_j}.
\end{equation}
For every $x\in X$, the corresponding fibre of $N_h(Y)$ contains the $h$-neighborhood
in $S$ of a nonempty set and hence has $m$-volume at least $\omega_mh^m$.  Integrating,
\begin{equation}\label{eq:nhy-lower}
 \abs{N_h(Y)}\gtrsim_n\abs Y
       \prod_{a_j<h}\frac h{a_j}.
\end{equation}
When $m=0$, this is simply $\abs{N_h(Y)}\ge\abs Y$.  On the other hand,
\eqref{eq:nhv-upper} gives
\[
 \frac{\abs Y}{\abs V}\abs{N_h(V)}
 \lesssim_n\abs Y\prod_{a_j<h}\frac h{a_j}.
\]
Together with \eqref{eq:nhy-lower}, this proves the result at every intermediate
scale.  If $Y$ is null, the right side of \eqref{eq:thickening} is zero.
\end{proof}

The same John-box argument, now using a box inside $E_V$ and a box inside the
Euclidean $h$-ball, gives the two-sided estimate
\begin{equation}\label{eq:parallel-volume}
 \abs{N_h(V)}\asymp_n\prod_{j=1}^n(a_j(V)+h).
\end{equation}

\begin{lemma}[Buffered parent stability]\label{lem:buffer}
Let $(W_\alpha)$ be a finite indexed family in $\R^n$, let
$0<r_\alpha\le C_0a_1(W_\alpha)$, and put
$W_\alpha^\sharp=N_{r_\alpha}(W_\alpha)$.  Then
\begin{equation}\label{eq:buffer-volume}
 \abs{W_\alpha^\sharp}\asymp_{n,C_0}\abs{W_\alpha},
 \qquad
 \Delta_{\max}(\{W_\alpha^\sharp\})\lesssim_{n,C_0}
       \Delta_{\max}(\{W_\alpha\}).
\end{equation}
\end{lemma}

\begin{proof}
The volume comparison follows from \eqref{eq:parallel-volume}.  If
$W_\alpha^\sharp\subset K$, then $W_\alpha\subset K$, and hence
\[
 \sum_{\alpha:W_\alpha^\sharp\subset K}\abs{W_\alpha^\sharp}
 \lesssim\sum_{\alpha:W_\alpha\subset K}\abs{W_\alpha}
 \le\Delta_{\max}(\{W_\alpha\})\abs K.
\]
Take the supremum over positive-volume convex $K$.  Every sum retains the parent
indices.
\end{proof}

The buffer lemma allows the coarse shading to live in a genuine neighborhood of the
parent while preserving both its volume scale and its indexed Katz--Tao density.
Section~\ref{sec:sharpness} explains why an intersection with the original parent
would not provide an aspect-uniform substitute.

\subsection{Shortest-direction projection}

Projection is useful only if the lost fibre direction has length comparable to the
analytic radius.  Choosing a shortest John direction of the parent provides this
comparison simultaneously for every child, regardless of the child's own
orientation.

\begin{lemma}[Projection--volume comparison]\label{lem:pv}
Let $e$ be a unit vector, let $\pi$ be orthogonal projection onto $e^\perp$, and set
\[
 P=N_h(V),\qquad Q=\pi(P).
\]
Then
\begin{equation}\label{eq:pv-general}
 c_nh\abs Q\le\abs P\le(\width_e(V)+2h)\abs Q.
\end{equation}
In particular, if $V\subset W$ and $\width_e(W)\le C_0h$, then
$\abs P\asymp_{n,C_0}h\abs Q$.
\end{lemma}

\begin{proof}
Let $K=\pi(V)$.  Orthogonal projection commutes with Minkowski addition by a ball,
so $Q=K+B_{e^\perp}(0,h)$.  If $z\in K+B(0,h/2)$, choose $x\in K$ with
$\abs{z-x}<h/2$ and a point $(t,x)\in V$ in coordinates $\R e\times e^\perp$.
The fibre over $z$ of the $h$-ball around $(t,x)$ has length at least $\sqrt3h$.
Thus
\[
 \abs P\ge\sqrt3h\abs{K+B(0,h/2)}.
\]
Convexity, with any $x_0\in K$, gives
$K+B(0,h)\subset x_0+2[(K+B(0,h/2))-x_0]$; hence the two projected volumes differ by
at most $2^{n-1}$, giving the lower bound.  The entire $e$-coordinate range of
$P$ has length at most $\width_e(V)+2h$, which proves the upper bound by Fubini.
\end{proof}

Let $W\subset\R^4$ have John semiaxes
$u_1\le u_2\le u_3\le u_4$, choose a shortest John direction $e_1$, and let
$h=c_0u_1$.  Since $W\subset4E_W$,
$\width_{e_1}(W)\le8u_1=8h/c_0$.  Therefore Lemma~\ref{lem:pv} applies uniformly
to every $V\subset W$, without any condition on the orientation of $V$.  Moreover,
in the same coordinates the product box
\begin{equation}\label{eq:parent-product-box}
 B=I\times B'
 =\prod_{j=1}^4[c_j-(4u_j+h),c_j+(4u_j+h)]
\end{equation}
contains $N_h(W)$ and satisfies
\begin{equation}\label{eq:product-box-scales}
 H:=\abs I\asymp_{c_0}h,\qquad
 R:=\operatorname{longside}(B')\asymp_{c_0}u_4,\qquad
 \abs B\asymp_{c_0}\abs W.
\end{equation}

\subsection{Mean-normalized weighted Frostman inheritance}

Thickening replaces the body mass $\abs V$ by $r_W(V)\abs V$.  The relative
Frostman coefficient must therefore be computed with the same weights.  The next two
lemmas compare this weighted coefficient with the original family and show that it
survives shortest-direction projection.

\begin{lemma}[Comparison with the unweighted coefficient]
\label{lem:weighted-frostman-comparison}
Let $\cV_W$ be nonempty and $C_F(\cV_W,W)\le C$.  Then
\begin{equation}\label{eq:weighted-frostman-comparison}
 C^{\rm wt}_{F,h}(\cV_W,W)
 \le C\frac{r_W^+}{\overline r_W}.
\end{equation}
\end{lemma}

\begin{proof}
For a convex $K\subset W$,
\[
 \sum_{V\subset K}r_W(V)\abs V
 \le r_W^+\sum_{V\subset K}\abs V
 \le r_W^+C\frac{T_W}{\abs W}\abs K.
\]
The weighted total density equals
\[
 \frac{\sum_Vr_W(V)\abs V}{\abs W}
 =\overline r_W\frac{T_W}{\abs W}.
\]
Division proves \eqref{eq:weighted-frostman-comparison}.  This concentration step
uses the maximum-to-mean ratio; the minimum inflation enters separately when shaded
density is converted in Corollary~\ref{cor:r4-unweighted}.
\end{proof}

\begin{lemma}[Weighted thickening and projection inheritance]
\label{lem:fr-inheritance}
Let $(V_i)$ be a nonempty indexed family in a convex body $W$, put
$P_i=N_h(V_i)\subset B$, assume $\abs B\le C_B\abs W$, and write
$r_i=\abs{P_i}/\abs{V_i}$,
$r_+=\max_i r_i$, and
$\overline r=(\sum_i r_i\abs{V_i})/(\sum_i\abs{V_i})$.  Then
\begin{equation}\label{eq:fr-thick}
 C_F(\{P_i\},B)
 \le C_B C^{\rm wt}_{F,h}(\{V_i\},W).
\end{equation}
If $C_F(\{V_i\},W)\le C$, then
\begin{equation}\label{eq:fr-thick-elementary}
 C_F(\{P_i\},B)
 \le C_BC\frac{r_+}{\overline r}.
\end{equation}
Suppose also that $B=I\times B'$, $Q_i=\pi(P_i)\subset B'$, and
\[
 c_Ph\abs{Q_i}\le\abs{P_i}\le C_Ph\abs{Q_i}.
\]
Then
\begin{equation}\label{eq:fr-project}
 C_F(\{Q_i\},B')
 \le\frac{C_P}{c_P}C_F(\{P_i\},B).
\end{equation}
\end{lemma}

\begin{proof}
Set
\[
 D_r=\frac{\sum_ir_i\abs{V_i}}{\abs W},
 \qquad D_P=\frac{\sum_i\abs{P_i}}{\abs B}.
\]
Then $D_P=D_r\abs W/\abs B\ge C_B^{-1}D_r$.  If
$P_i\subset K\subset B$, then $V_i\subset K\cap W$.  If the relevant sum is
nonzero, $K\cap W$ is a positive-volume convex body, and Definition
\ref{def:weighted-frostman} gives
\[
 \sum_{P_i\subset K}\abs{P_i}
 =\sum_{P_i\subset K}r_i\abs{V_i}
 \le C^{\rm wt}_{F,h}D_r\abs{K\cap W}
 \le C^{\rm wt}_{F,h}D_r\abs K.
\]
The zero-sum case is immediate.  Division by $D_P\abs K$ proves
\eqref{eq:fr-thick}; Lemma~\ref{lem:weighted-frostman-comparison} gives
\eqref{eq:fr-thick-elementary}.

For projection, $Q_i\subset K'\subset B'$ implies
$P_i\subset I\times K'$.  With $H=\abs I$,
\[
 \sum_{Q_i\subset K'}\abs{Q_i}
 \le\frac{C_F(\{P_i\},B)D_PH\abs{K'}}{c_Ph},
\qquad
 D_Q:=\frac{\sum_i\abs{Q_i}}{\abs{B'}}
 \ge\frac{D_PH}{C_Ph}.
\]
After division, the factors $H/h$ cancel exactly, proving
\eqref{eq:fr-project}.
\end{proof}

\subsection{The nonuniform local theorem}

We can now keep the full inflation profile in the local response.  The theorem below
is stated for a fixed base family so that its weighted coefficient and denominators
remain available uniformly after any later refinement.

\begin{theorem}[Intrinsic nonuniform local transfer]\label{thm:r4-local}
Let $W\subset\R^4$ have John semiaxes
$u_1\le u_2\le u_3\le u_4$, fix $c_0>0$, and put $h=c_0u_1$.
Let $\cV_W=(V_i)_{i\in J}$ be an arbitrary nonempty finite indexed family of
convex bodies $V_i\subset W$.  Write
$r_i=\abs{N_h(V_i)}/\abs{V_i}$ and
$\overline r=(\sum_i r_i\abs{V_i})/(\sum_i\abs{V_i})$.
For measurable $Y_i\subset V_i$, use
\eqref{eq:local-densities} and Definition~\ref{def:weighted-frostman}, and put
\[
 \delta_W=\frac{u_1}{u_4},\qquad
 \Lfac_W=2+\log(2+\#J)+\log(2+u_4/u_1).
\]
For every $\eps>0$ there are $K_\eps\ge2$ and
$A_{0,\eps}<\infty$ such that
\begin{equation}\label{eq:r4-local}
 \abs{N_h(\bigcup_iY_i)}
 \gtrsim_{\eps,c_0}
 \Lfac_W^{-A_{0,\eps}}\delta_W^\eps
 \bigl(C^{\rm wt}_{F,h}(\cV_W,W)\bigr)^{-1}
 \bigl(d^{\rm wt}_{W,h}(Y)\bigr)^{K_\eps}\abs W.
\end{equation}
The estimate is uniform over every measurable refinement of the fixed base family;
the quantities $r_i$, $\overline r$, $C^{\rm wt}_{F,h}$, and $T_W$ remain frozen.
\end{theorem}

\begin{proof}
Set $P_i=N_h(V_i)$ and $S_i=N_h(Y_i)$.  For a finite family,
\begin{equation}\label{eq:union-neighborhood-identity}
 \bigcup_iS_i=N_h(\bigcup_iY_i).
\end{equation}
Lemma~\ref{lem:thickening} gives
\[
 \sum_i\abs{S_i}\gtrsim\sum_ir_i\abs{Y_i}.
\]

Use the shortest-direction product box \eqref{eq:parent-product-box}, let $\pi$ be
projection onto $e_1^\perp$, and put $Q_i=\pi(P_i)$.  Then
$S_i\subset P_i\subset I\times Q_i$ and
\begin{equation}\label{eq:pi-volume}
 \abs{P_i}\asymp_{c_0}h\abs{Q_i}.
\end{equation}
Equations \eqref{eq:fr-thick} and \eqref{eq:fr-project} give
\begin{equation}\label{eq:projected-frostman}
 C_F(\{Q_i\},B')\lesssim_{c_0}
 C^{\rm wt}_{F,h}(\cV_W,W).
\end{equation}
Each $Q_i$ contains a three-ball of radius $h$ and has diameter $O(R)$.

Apply Theorem~\ref{thm:lf3} to arbitrary measurable $E_i\subset Q_i$.  If its
exponent is less than $2$, replace it by $2$: since its density lies in $[0,1]$,
this only weakens the lower bound.  The explicit cardinality and projected-axis
losses are bounded below by $\Lfac_W^{-A_{0,\eps}}$.  Theorem~\ref{thm:lift},
with $\Psi(t)=t^{K_\eps}$, then gives
\begin{equation}\label{eq:local-lifted}
 \abs{\bigcup_iS_i}\gtrsim_{\eps,c_0}
 H\abs{B'}\,\Lfac_W^{-A_{0,\eps}}
 \left(\frac hR\right)^\eps
 \bigl(C^{\rm wt}_{F,h}\bigr)^{-1}
 \left(\frac{\sum_i\abs{S_i}}
 {H\sum_i\abs{Q_i}}\right)^{K_\eps}.
\end{equation}
By \eqref{eq:pi-volume} and $H\asymp h$,
\[
 \frac{\sum_i\abs{S_i}}{H\sum_i\abs{Q_i}}
 \gtrsim
 \frac{\sum_ir_i\abs{Y_i}}{\sum_ir_i\abs{V_i}}
 =d^{\rm wt}_{W,h}(Y).
\]
Finally, $H\abs{B'}\asymp\abs W$ and $h/R\asymp u_1/u_4$ by
\eqref{eq:product-box-scales}; use \eqref{eq:union-neighborhood-identity}.
Since every $Q_i$ contains an $h$-ball and lies in a box of longest side
$O(u_4)$, the dyadic axis ranges are controlled by $u_4/h$.  This gives the stated
local logarithm without introducing a child minimum scale.
\end{proof}

\begin{corollary}[Unweighted local efficiencies]\label{cor:r4-unweighted}
In the setting of Theorem~\ref{thm:r4-local},
\begin{equation}\label{eq:r4-local-kappa}
 \abs{N_h(\bigcup_iY_i)}
 \gtrsim_{\eps,c_0}
 \Lfac_W^{-A_{0,\eps}}\delta_W^\eps
 \kappa_W^{(K_\eps)}d_W(Y)^{K_\eps}\abs W.
\end{equation}
If $C_F(\cV_W,W)\le C$, then
\begin{equation}\label{eq:r4-local-eta}
 \abs{N_h(\bigcup_iY_i)}
 \gtrsim_{\eps,c_0}
 \Lfac_W^{-A_{0,\eps}}\delta_W^\eps
 C^{-1}\eta_W^{(K_\eps)}d_W(Y)^{K_\eps}\abs W.
\end{equation}
\end{corollary}

\begin{proof}
Insert \eqref{eq:weighted-unweighted-density} into
\eqref{eq:r4-local} to obtain \eqref{eq:r4-local-kappa}.  Lemma
\ref{lem:weighted-frostman-comparison} gives
\[
 \bigl(C^{\rm wt}_{F,h}\bigr)^{-1}
 \left(\frac{r_W^-}{\overline r_W}\right)^{K_\eps}
 \ge
 C^{-1}\frac{\overline r_W}{r_W^+}
 \left(\frac{r_W^-}{\overline r_W}\right)^{K_\eps},
\]
which is \eqref{eq:r4-local-eta}.
\end{proof}

The two factors in the elementary coefficient have different origins.  Weighted
Frostman inheritance contributes $\overline r_W/r_W^+$ once, whereas conversion from
the original shaded density to the weighted density contributes
$r_W^-/\overline r_W$ with the response exponent $K_\eps$.  Keeping these mechanisms
separate yields the mean-normalized expression in \eqref{eq:r4-local-eta}.

\section{Weighted profile assembly and cellular factoring}\label{sec:cellular}

The local theorem supplies a response for each parent, but the child loads $D_W$ and
the local efficiencies $\alpha_W$ need not be comparable.  Selecting a common load
level would change the coarse normalization of the main theorem.  The weighted
H\"older inequality below instead measures the cost of assembling all parents in a
single explicit coefficient.  The cellular construction then arranges that the
local estimates, fine multiplicities, and coarse multiplicity all refer to one
refinement.

\subsection{The weighted profile inequality}

The coefficient $\mathfrak P_K$ compares the global body-mass-weighted shaded
density with the sum of the parentwise responses $\alpha_Wd_W^K$.  Its denominator
penalizes precisely those parents for which a large load is paired with a weak local
efficiency.

\begin{lemma}[Profile H\"older inequality]\label{lem:profile-holder}
Let $K>1$ and $q=K/(K-1)$.  On a finite index set, suppose
$p_W>0$, $\sum_Wp_W=1$, $D_W>0$, $\alpha_W>0$, and $d_W\ge0$.
Then
\begin{equation}\label{eq:profile-holder}
 \sum_Wp_W\alpha_Wd_W^K
 \ge
 \mathfrak P_K(\mathbf D,\boldsymbol\alpha;\mathbf p)
 \left(
  \frac{\sum_Wp_WD_Wd_W}{\sum_Wp_WD_W}
 \right)^K.
\end{equation}
Moreover,
\begin{equation}\label{eq:profile-basic-bounds}
 0<\mathfrak P_K\le\max_W\alpha_W.
\end{equation}
If $D_{\max}/D_{\min}\le\Lambda_D$ and $\alpha_W\ge\alpha_0$, then
\begin{equation}\label{eq:profile-balanced}
 \mathfrak P_K\ge\Lambda_D^{-K}\alpha_0.
\end{equation}
If $D_W\equiv D_0$ and $\alpha_W\equiv\alpha_0$, then
$\mathfrak P_K=\alpha_0$.
\end{lemma}

\begin{proof}
Write
\[
 \sum_Wp_WD_Wd_W
 =
 \sum_Wp_W
 (\alpha_W^{1/K}d_W)(D_W\alpha_W^{-1/K}).
\]
H\"older gives
\[
 \sum_Wp_WD_Wd_W
 \le
 \left(\sum_Wp_W\alpha_Wd_W^K\right)^{1/K}
 \left(\sum_Wp_WD_W^q
              \alpha_W^{-1/(K-1)}\right)^{1/q}.
\]
Rearranging gives \eqref{eq:profile-holder}.

Let $\alpha_*=\max_W\alpha_W$.  Since
$\alpha_W^{-1/(K-1)}\ge\alpha_*^{-1/(K-1)}$ and
$\sum p_WD_W^q\ge(\sum p_WD_W)^q$, the denominator in
\eqref{eq:profile-def} is at least
$\alpha_*^{-1}(\sum p_WD_W)^K$, and hence
\eqref{eq:profile-basic-bounds}; in particular $\mathfrak P_K\le1$ when
$\alpha_W\le1$.  Under the balanced hypotheses, the numerator is at least
$D_{\min}^K$, while the denominator is at most
$D_{\max}^K\alpha_0^{-1}$, proving \eqref{eq:profile-balanced}.  Direct
substitution gives the constant-profile identity.
\end{proof}

\begin{remark}[Equality and optimal scalar coefficient]\label{rem:profile-equality}
Equality in \eqref{eq:profile-holder} holds precisely, on the positive indices, when
\begin{align*}
 \alpha_Wd_W^{K-1}&\ \text{is proportional to}\ D_W,\\
 d_W&\ \text{is proportional to}\
 \left(\frac{D_W}{\alpha_W}\right)^{1/(K-1)}.
\end{align*}
Thus $\mathfrak P_K$ is the optimal coefficient for this scalar H\"older assembly
inequality.  This optimality concerns the assembly step; the geometric argument may
carry additional losses.
\end{remark}

\subsection{A common cellular refinement}

We next isolate the global construction.  The input is a family of parentwise local
responses whose coefficients were fixed before refinement.  Four mass-weighted
restrictions are used: the first two regularize the inner and global fine
multiplicities, the third regularizes the indexed number of active parents in a
cell, and the fourth regularizes the fine mass of a cell.  A null cleanup between
the second and third restrictions makes positive-measure cell activity compatible
with pointwise fibres.  Since every later restriction is by a common union of whole
half-open cells, it preserves the multiplicity information obtained earlier.

Let $\cV=\disjunion_{W\in\cW}\cV_W$ be finite indexed families of bounded bodies
in $\R^4$, with $Y(V)\subset V\subset W$.  Assume every parent group has positive
initial shaded mass; initial zero-mass groups may be removed before invoking the
theorem.  Write
\[
 M_0=\sum_{V\in\cV}\abs{Y(V)},\qquad
 T_0=\sum_{V\in\cV}\abs V,\qquad
 T_W=\sum_{V\in\cV_W}\abs V,\qquad
 D_W=\frac{T_W}{\abs W}.
\]
The $T_W$ and $D_W$ remain frozen through every refinement.  Suppose the parent
volumes and ordered John semiaxes are uniformly comparable to fixed scales
$w_1\le\cdots\le w_4$.  Put
\[
 \ell=h=w_1/100,\qquad W^\sharp=N_{w_1}(W),
\]
and assume $\abs{W^\sharp}\asymp\abs W$.

For each parent fix, before any refinement, a number $\alpha_W\in(0,1]$.
Assume that for every measurable refinement and every parent with positive refined
mass,
\begin{equation}\label{eq:cellular-local-input}
 \abs{N_h(U'_W)}
 \ge\Gamma\alpha_Wd_W^K\abs W,\qquad
 d_W=\frac{\sum_{V\in\cV_W}\abs{Y'(V)}}{T_W},
 \qquad K>1,
\end{equation}
with common $\Gamma>0$ and $K$.  Let $N=\#\cV$, $P=\#\cW$, and
\begin{equation}\label{eq:cell-log}
 \mathcal G=
 2+\log(2+N)+\log(2+P)+\log(2+w_4/w_1).
\end{equation}

\begin{lemma}[Grid-neighbor covering]\label{lem:grid-neighbor-covering}
Let $\mathcal Q_\ell$ be the half-open grid of side $\ell$ in $\R^d$, and let
$U$ be measurable such that the grid cube containing every point of $U$ has
positive-measure intersection with $U$.  If
\[
 Z_\ell(U)=\bigcup\{Q\in\mathcal Q_\ell:\abs{U\cap Q}>0\},
\]
then
\begin{equation}\label{eq:grid-covering-lemma}
 \abs{N_h(U)}\le C_d(1+h/\ell)^d\abs{Z_\ell(U)},
\end{equation}
and $Z_\ell(U)\subset N_{\sqrt d\,\ell}(U)$.
\end{lemma}

\begin{proof}
If $y\in N_h(U)$, choose $x\in U$ with $\abs{x-y}<2h$.  The lattice indices of
their half-open grid cubes differ in each coordinate by at most
$2+\lceil2h/\ell\rceil$.  The cube of $x$ is active, so the cube of $y$ lies among
at most $C_d(1+h/\ell)^d$ neighbors of an active cube.  Disjointness and equal
grid volume prove the first assertion.  If $Q$ is active, every point of $Q$ is
within $\sqrt d\,\ell$ of a point of $U\cap Q$.
\end{proof}

\begin{theorem}[Profile cellular factoring]\label{thm:cellular}
Under the preceding hypotheses, there are a single refinement $Y'$, coarse shadings
$Z(W^\sharp)$ for every parent index, a nonempty positive final-mass set
\[
 \cW_+=\{W:\ M'_W=\sum_{V\in\cV_W}\abs{Y'(V)}>0\},
\]
positive dyadic $a,b,c,M$, and $\sigma=M/c>0$ such that
$Z(W^\sharp)=\varnothing$ for $W\notin\cW_+$ and
\begin{equation}\label{eq:cell-mass}
 \sum_V\abs{Y'(V)}\gtrsim\mathcal G^{-4}M_0.
\end{equation}
The containment \eqref{eq:main-compat}, the pointwise conclusions
\eqref{eq:main-inner}--\eqref{eq:main-coarse-mult}, the fibre equality
\eqref{eq:main-fibre}, and local uniformity
\eqref{eq:main-local-uniformity} hold with the full frozen coarse family
$\cW^\sharp$.  For every $W\in\cW_+$,
\begin{equation}\label{eq:cell-product}
 \mu(\cV,Y)\lesssim\mathcal G^4
 \mu(\cW^\sharp,Z)\mu(\cV_W,Y'),
\end{equation}
and
\begin{equation}\label{eq:cell-coarse-density}
 \lambda(\cW^\sharp,Z)
 \gtrsim
 \Gamma\mathcal G^{-4K}
 \mathfrak P_K(\mathbf D,\boldsymbol\alpha;\mathbf p)
 \lambda(\cV,Y)^K,
\end{equation}
where
\[
 p_W=\frac{\abs W}{\sum_{W'}\abs{W'}}.
\]
\end{theorem}

\begin{proof}
We construct the refinement through four nested mass-weighted selections.  The first
two act directly on the fine multiplicities.  After removing null cell pieces, the
third records how many parent indices are active in a cell, and the fourth makes the
fine mass comparable across the retained cells.  Each selection is charged to the
current fine shaded mass, while the common restrictions preserve all uniformities
already obtained.

\smallskip
\noindent\emph{Inner multiplicity.}
For the current shadings define
\[
 m_W(x)=\sum_{V\in\cV_W}\ind_{Y(V)}(x).
\]
The identity $\sum_W\int m_W=\sum_V\abs{Y(V)}$ and
$1\le m_W\le N$ on the fine union permit a dyadic choice of $a$ carrying a
logarithmic fraction of the fine mass.  Intersecting every shading in a parent with
$\{a\le m_W<2a\}$ makes the inner multiplicity uniform on the surviving parent
union.

\smallskip
\noindent\emph{Global fine multiplicity.}
For the refined shading let $m=\sum_Wm_W$.  Select one dyadic $c$ by fine mass
and intersect every shading with the common set $\{c\le m<2c\}$.  Because the same
set is used for every child, the inner count obtained in the preceding step remains
valid at every surviving point.

\smallskip
\noindent\emph{Positive-measure cell activity.}
Let $U_W^{(2)}$ be the parent union after these selections and partition $\R^4$
into half-open cubes $Q$ of side $\ell$.  Remove from every child shading the same
null set
\begin{equation}\label{eq:null-cleanup}
 E=\bigcup_{\substack{W,Q\\\abs{U_W^{(2)}\cap Q}=0}}
       (U_W^{(2)}\cap Q).
\end{equation}
Only finitely many cells meet the bounded parents, so $E$ is null and the previous
mass and multiplicity conclusions are unchanged.  After removing it from every
child, any surviving fine point makes its assigned parent active in its cell in the
positive-measure sense required by the grid argument.

\smallskip
\noindent\emph{Indexed active-parent multiplicity.}
Call $W$ active in $Q$ if $\abs{U_W^{(2)}\cap Q}>0$, and let $b_Q$ be the
number of active parent indices.  On every positive-fine-mass cell,
$1\le b_Q\le P$.  Select a dyadic $b$ by the fine cell masses
\[
 M_Q=\sum_V\abs{Y(V)\cap Q},
\]
and retain the cells with $b\le b_Q<2b$.  This whole-cell restriction preserves the
two fine multiplicity ranges and makes the indexed coarse count uniform on every
retained cell.

\smallskip
\noindent\emph{Cell fine mass.}
Let $A=\sum_QM_Q>0$ after the third selection and let $J$ be the number of cells
with $M_Q>0$.  Discard cells with
\begin{equation}\label{eq:relative-cell-cutoff}
 M_Q<\frac{A}{2J}.
\end{equation}
This loses less than $A/2$.  The remaining positive masses have ratio at most
$2J$, so one dyadic level
\begin{equation}\label{eq:M-level}
 M\le M_Q<2M
\end{equation}
carries a logarithmic fraction of the mass.  Intersect every child shading with
the same union of selected whole cells.  The final cell masses are therefore
comparable without disturbing any of the three multiplicity ranges.

The number of cells satisfies
\begin{equation}\label{eq:J-bound}
 J\lesssim P\prod_{j=1}^4(1+w_j/w_1).
\end{equation}
Indeed, a cube meeting $W$ lies in $N_{2\ell}(W)$, and disjointness of the grid
and the John-box estimate give
\[
 \#\{Q:Q\cap W\ne\varnothing\}\ell^4
 \le\abs{N_{2\ell}(W)}
 \lesssim\prod_{j=1}^4(w_j+\ell).
\]
Thus each of the four selections costs at most a factor controlled by
$\mathcal G$, which gives \eqref{eq:cell-mass}.  The construction has used the
pointwise and cellular quantities, while retaining the complete range of parent
loads for the profile assembly.

Let $\cW_+$ be the parents with positive final mass.  For every parent, including
those outside $\cW_+$, define
\begin{equation}\label{eq:coarse-shading-definition}
 Z(W^\sharp)=
 \bigcup\{Q:Q\text{ is finally selected and }\abs{U'_W\cap Q}>0\}.
\end{equation}
This is empty for $W\notin\cW_+$.  If a parent was active in a finally selected
cell, the whole-cell restriction preserves its positive-measure fine piece, so it
belongs to $\cW_+$.  Consequently the coarse multiplicity on a selected cell is
the previously selected $b_Q\in[b,2b)$.  Parents outside $\cW_+$ contribute empty
coarse shadings but remain in the fixed denominator.  Since a grid cube has diameter
$2\ell<w_1$,
\[
 U'_W\subset Z(W^\sharp)
 \subset N_{w_1}(U'_W)\cap W^\sharp.
\]
The common whole-cell restrictions preserve the inner and global counts.

At a final fine point, let $p(x)$ count assigned parent indices contributing a child.
Then
\[
 c\le m_{\cV,Y'}(x)<2a\,p(x),\qquad p(x)\le b_Q<2b,
\]
and hence
\begin{equation}\label{eq:c-ab}
 c<4ab.
\end{equation}
For $W\in\cW_+$ and for the full coarse family,
\[
 a\le\mu(\cV_W,Y')<2a,\qquad
 b\le\mu(\cW^\sharp,Z)<2b.
\]
Writing $M'=\sum_V\abs{Y'(V)}$ and $U'=U(\cV,Y')$, we have
\[
 \mu(\cV,Y)
 \le\frac{M_0}{\abs{U'}}
 \lesssim\mathcal G^4\frac{M'}{\abs{U'}}
 <2\mathcal G^4c.
\]
Together with \eqref{eq:c-ab}, this proves \eqref{eq:cell-product}.

Unique half-open cell ownership, the global null cleanup, and the disjoint assigned
index union give
\[
 (\cV)_{Y'}(x)=
 \disjunion_{\substack{W\in\cW\\x\in Z(W^\sharp)}}
 (\cV_W)_{Y'}(x).
\]
On a selected cell,
\[
 c\abs{U'\cap Q}\le M_Q<2c\abs{U'\cap Q},
 \qquad M\le M_Q<2M,
\]
so $\abs{U'\cap Q}\asymp\sigma=M/c$.  A ball of radius $w_1$ centered
in the coarse union contains its center cell and meets only $O(1)$ grid cells,
proving local uniformity.

For the density conclusion, the null cleanup gives the pointwise hypothesis of
Lemma~\ref{lem:grid-neighbor-covering}, and
$Z(W^\sharp)=Z_\ell(U'_W)$.  Since $h=\ell$,
\begin{equation}\label{eq:grid-neighbor}
 \abs{Z(W^\sharp)}\gtrsim\abs{N_h(U'_W)}.
\end{equation}
Put $d_W=0$ for $W\notin\cW_+$.  The local input, summation, and
$\abs{W^\sharp}\asymp\abs W$ give
\begin{equation}\label{eq:profile-density-sum}
 \lambda(\cW^\sharp,Z)
 \gtrsim\Gamma\sum_Wp_W\alpha_Wd_W^K.
\end{equation}
Moreover,
\[
 \sum_Wp_WD_Wd_W
 =\frac{M'}{\sum_W\abs W},
 \qquad
 \sum_Wp_WD_W
 =\frac{T_0}{\sum_W\abs W}.
\]
Thus
\begin{equation}\label{eq:profile-density-normalization}
 \frac{\sum_Wp_WD_Wd_W}{\sum_Wp_WD_W}
 =\frac{M'}{T_0}
 \gtrsim\mathcal G^{-4}\lambda(\cV,Y).
\end{equation}
Apply Lemma~\ref{lem:profile-holder} to
\eqref{eq:profile-density-sum} and use
\eqref{eq:profile-density-normalization}, which yields
\eqref{eq:cell-coarse-density}; the four-selection loss is raised to $K$.
\end{proof}

\subsection{From the local response to the main theorem}\label{subsec:assembly}

It remains to match the local geometric response from Section~\ref{sec:local-r4}
with the abstract cellular theorem.  The initial reduction to $\cW_0$ provides the
positive-mass hypothesis, parent comparability makes the local constants uniform,
and the fixed local efficiencies become the coefficients in the H\"older profile.

\begin{proof}[Proof of Theorem~\ref{thm:main}]
Restrict first to the parent groups in $\cW_0$ from
\eqref{eq:initial-parent-support}.  Every discarded input shading is null; finiteness
therefore preserves both the shaded mass and the measure of the fine union.
Let $\cV_0$ be the remaining assigned child family and set
\[
 \lambda_0=\frac{\sum_{V\in\cV_0}\abs{Y(V)}}
                  {\sum_{V\in\cV_0}\abs V}.
\]
Then $\lambda_0\ge\lambda$.

Fix $\ell=h=w_1/100$.  For every $W\in\cW_0$,
\[
 \frac{h}{a_1(W)}
 \in\left[\frac1{100\Lambda_W},\frac{\Lambda_W}{100}\right],
 \qquad
 \frac{a_1(W)}{a_4(W)}
 \asymp_{\Lambda_W}\frac{w_1}{w_4}.
\]
The local logarithm is bounded by $\Lfac$, uniformly in $W$.  Corollary
\ref{cor:r4-unweighted}, in its intrinsic form, supplies
\eqref{eq:cellular-local-input} for every later measurable refinement with
\[
 \Gamma\gtrsim_{\eps,\Lambda_W}
 \Lfac^{-A_{0,\eps}}\left(\frac{w_1}{w_4}\right)^\eps,
 \qquad
 \alpha_W=\kappa_W^{(K_\eps)},\qquad K_\eps\ge2.
\]
The coefficient $\alpha_W\in(0,1]$ is determined before the cellular refinement.
Lemma~\ref{lem:buffer} gives
$\abs{W^\sharp}\asymp_{\Lambda_W}\abs W$ and the required Katz--Tao stability.

Apply Theorem~\ref{thm:cellular} to $\cW_0$ and $\cV_0$, and extend its output to
the original indexed family by setting $Y'(V)=\varnothing$ for
$V\in\cV\setminus\cV_0$.  The extension is a refinement.  Since the discarded input
shadings form a finite null union, $\mu(\cV,Y)=\mu(\cV_0,Y)$; the empty outputs also
remove those indices from every refined pointwise fibre.  With
$\alpha_W=\kappa_W^{(K_\eps)}$, the cellular profile is
\eqref{eq:intrinsic-main-profile}.  Since
$\lambda_0^{K_\eps}\ge\lambda^{K_\eps}$, its density estimate proves
\eqref{eq:main-coarse-density}.  Its mass, compatibility, multiplicity, product,
fibre, and local-uniformity conclusions prove
\eqref{eq:main-mass}--\eqref{eq:main-local-uniformity} after increasing one
$A_\eps$ to dominate the local and four-selection logarithmic powers.  Finally,
\[
 \Delta_{\max}(\cW_0^\sharp)
 \lesssim_{\Lambda_W}\Delta_{\max}(\cW_0)
 \le\Delta_{\max}(\cW),
\]
which gives \eqref{eq:main-kt} and completes the proof.
\end{proof}

\begin{proof}[Proof of Corollary~\ref{cor:main-elementary}]
Lemma~\ref{lem:weighted-frostman-comparison} and
\eqref{eq:kappa-def}--\eqref{eq:eta-def} give, parentwise,
\[
 \kappa_W^{(K_\eps)}\ge C^{-1}\eta_W^{(K_\eps)}.
\]
Substitution into the denominators of
\eqref{eq:intrinsic-main-profile} and
\eqref{eq:elementary-main-profile} yields
\[
 \mathfrak P_\eps^{\rm int}
 \ge C^{-1}\mathfrak P_\eps^{\rm elem}.
\]
Theorem~\ref{thm:main} now gives \eqref{eq:elementary-main-density} with the
same refinement and all the same structural conclusions.
\end{proof}

\section{Consequences and the comparable-scale regime}
\label{sec:profile-corollaries}

The profile theorem becomes more familiar when the parent loads and local inflation
efficiencies are balanced.  We first record this direct consequence and use a mixed
geometry example to show that inflation balance is weaker than coordinatewise axis
comparability.  We then recover the comparable-scale conclusion by selecting a load
level before applying the main theorem.

\subsection{Balanced profiles and mixed child geometry}

\begin{corollary}[Balanced elementary profile]\label{cor:balanced-profile}
In Corollary~\ref{cor:main-elementary}, suppose
\[
 \frac{\max_{W\in\cW_0}D_W}{\min_{W\in\cW_0}D_W}
 \le\Lambda_D,
 \qquad
 \eta_W^{(K_\eps)}\ge\eta_0>0
 \quad(W\in\cW_0).
\]
Then the same refinement satisfies
\begin{equation}\label{eq:balanced-profile-density}
 \lambda(\cW_0^\sharp,Z)
 \gtrsim_{\eps,\Lambda_W}
 \Lfac^{-A_\eps}\left(\frac{w_1}{w_4}\right)^\eps
 C^{-1}\Lambda_D^{-K_\eps}\eta_0\lambda^{K_\eps}.
\end{equation}
\end{corollary}

\begin{proof}
Apply \eqref{eq:profile-balanced} to
$\alpha_W=\eta_W^{(K_\eps)}$ in
\eqref{eq:elementary-main-profile}.
\end{proof}

The parallel-volume estimate \eqref{eq:parallel-volume} gives the useful criterion
\begin{equation}\label{eq:inflation-product}
 r_W(V)\asymp
 \prod_{j=1}^4\left(1+\frac{h}{a_j(V)}\right).
\end{equation}
Thus, if the products on the right are uniformly comparable among the children in a
fixed parent, then
$r_W^-\asymp\overline r_W\asymp r_W^+$ and
$\eta_W^{(K_\eps)}\asymp1$.  This product condition allows coordinatewise axis
ratios to be unbounded.

\begin{example}[Mixed sub-buffer shapes]\label{ex:mixed-subbuffer}
Let $\tau\ll h$ and take equal-volume convex children whose John-axis quadruples are
\[
 (\tau^3,\tau,\tau,\tau)
 \quad\text{or}\quad
 (\tau^2,\tau^2,\tau,\tau).
\]
Both products of the four axes equal $\tau^6$, whereas the first two coordinatewise
axis ratios between the two shapes are $\tau$ and $\tau^{-1}$.  Hence no fixed
coordinatewise comparability constant survives as $\tau\to0$.  Since all axes are
$o(h)$,
\[
 \abs{N_h(V)}\asymp h^4,\qquad
 r_W(V)\asymp\frac{h^4}{\abs V}.
\]
Equal child volumes therefore make all inflation ratios comparable and give
$\eta_W^{(K_\eps)}\asymp1$.  Whenever the parent loads are also balanced,
Corollary~\ref{cor:balanced-profile} gives a nondegenerate conclusion although no
uniform coordinatewise child-axis comparison is available.
\end{example}

\subsection{The comparable-scale specialization}

\begin{corollary}[Comparable-scale buffered factoring]\label{cor:comparable-scale}
Assume the hypotheses of Corollary~\ref{cor:main-elementary} and, in addition, fix
$\Lambda_V\ge1$ such that for every pair of child indices $V,V'$ and every
$1\le j\le4$,
\[
 \Lambda_V^{-1}a_j(V')\le a_j(V)\le\Lambda_Va_j(V').
\]
Then there is a nonempty selected parent subfamily $\cW'\subset\cW$ for which all
the structural conclusions of Theorem~\ref{thm:main} hold, every retained coarse
index has positive final mass, and
\begin{equation}\label{eq:comparable-scale-density}
 \lambda(\cW'^\sharp,Z)
 \gtrsim_{\eps,\Lambda_V,\Lambda_W}
 \Lfac^{-A_\eps}\left(\frac{w_1}{w_4}\right)^\eps
 C^{-1}\lambda^{K_\eps}.
\end{equation}
In particular, the conclusion includes the one-refinement, multiplicity-product,
fibre, local-uniformity, and Katz--Tao properties of Theorem~\ref{thm:main}, now on
the selected parent family.
\end{corollary}

\begin{proof}
Discard first the initial zero-shaded-mass parent groups.  This preserves the numerator
of the fine density and decreases its denominator.  Coordinatewise child-axis
comparability and the John-volume estimate imply that all child volumes are
comparable to one value $v$.  Parent-axis comparability makes all parent volumes
comparable as well.  If $n_W=\#\cV_W$, then
\begin{equation}\label{eq:comparable-load-count}
 D_W=\frac{T_W}{\abs W}\asymp_{\Lambda_V,\Lambda_W}
 n_W\frac v{\abs W}.
\end{equation}
For a nonempty group, $1\le n_W\le N=\#\cV$, so the positive loads occupy
$O_{\Lambda_V,\Lambda_W}(1+\log(2+N))$ dyadic levels.

Choose a load level $D\le D_W<2D$ carrying a reciprocal-logarithmic fraction of
the original shaded mass, and call its parent family $\cW_1$.  If $M_1,T_1$ are
the selected shaded and body masses, let
$\cV_1=\disjunion_{W\in\cW_1}\cV_W$ and
$U_1=U(\cV_1,Y)$ and $U_{\rm full}=U(\cV,Y)$.  Then
\[
 M_1\gtrsim\Lfac^{-1}M_0,\qquad
 \lambda_1=\frac{M_1}{T_1}
 \ge\frac{M_1}{T_0}
 \gtrsim\Lfac^{-1}\lambda.
\]
The selected loads are $2$-comparable.  Since $U_1\subset U_{\rm full}$, the
retained-mass estimate also compares the original and selected average
multiplicities:
\[
 \mu(\cV,Y)=\frac{M_0}{\abs{U_{\rm full}}}
 \lesssim \Lfac\frac{M_1}{\abs{U_{\rm full}}}
 \le \Lfac\frac{M_1}{\abs{U_1}}
 =\Lfac\mu(\cV_1,Y).
\]

At the common radius $h=w_1/100$, \eqref{eq:inflation-product} and the global
coordinatewise child-axis hypothesis imply
\[
 r_W^-\asymp_{\Lambda_V}\overline r_W
 \asymp_{\Lambda_V}r_W^+,
 \qquad
 \eta_W^{(K_\eps)}
 \gtrsim_{\Lambda_V,K_\eps}1
\]
for every $W\in\cW_1$.  Lemma~\ref{lem:profile-holder}, or
Corollary~\ref{cor:balanced-profile}, therefore bounds the selected elementary
profile below by a positive constant depending only on
$\Lambda_V,\Lambda_W,\eps$.

Apply Corollary~\ref{cor:main-elementary} to the selected family
$(\cW_1,\cV_1)$, and extend its output by $Y'(V)=\varnothing$ for
$V\in\cV\setminus\cV_1$.  Its mass conclusion, combined with
$M_1\gtrsim\Lfac^{-1}M_0$, gives the mass bound relative to the original family
after increasing $A_\eps$.  Its density conclusion contains
$\lambda_1^{K_\eps}\gtrsim\Lfac^{-K_\eps}\lambda^{K_\eps}$, and hence yields
\eqref{eq:comparable-scale-density} after the same increase.
The displayed comparison of average multiplicities and the selected-family product
give the full-family multiplicity product with one further factor $\Lfac$, which is
absorbed by the same increase of $A_\eps$.  The empty extension makes all full-family
pointwise multiplicity and fibre statements literal.

The profile theorem keeps final zero-mass parent buffers in the denominator fixed by
the selected input family.
Delete precisely those empty-shading coarse indices.  The coarse numerator, coarse
union, all pointwise multiplicities, and fibre equality are unchanged; coarse density
can only increase and $\Delta_{\max}$ can only decrease.  The remaining parent
family $\cW'$ consists precisely of positive final-mass parents, so every average
multiplicity in the corollary is defined.  All structural conclusions are therefore
inherited by $\cW'$.
\end{proof}

\begin{remark}
Corollary~\ref{cor:comparable-scale} normalizes the coarse family after selecting a
parent-load level.  Theorem~\ref{thm:main}, by contrast, keeps the entire initial
positive-mass support and records load imbalance through its profile.  The two
conclusions agree in the balanced regime but use different denominators before that
specialization.
\end{remark}

\section{Structural limitations and model examples}\label{sec:sharpness}

The definitions in the main theorem reflect several distinct geometric and
normalization constraints.  The examples below explain them one at a time: the
external buffer, indexed counting, mean-normalized inflation, the parent-load
profile, the projection direction, and the two cellular discretizations.  Each model
addresses a particular mechanism rather than asserting global optimality of the
full theorem.

\subsection{Why the coarse shading lives in the external buffer}

Fix $L\gg r>0$ and, in $\R\times\R^3$, consider the cone
\begin{equation}\label{eq:cone-parent}
 W=\{(t,y):0\le t\le L,\ \abs y\le (r/L)t\}.
\end{equation}
Its shortest John scale is comparable to $r$.  At the apex,
\begin{equation}\label{eq:cone-volume}
 \abs{B(0,r)\cap W}\asymp\int_0^r\left(\frac{rt}{L}\right)^3dt
 \asymp\frac{r^7}{L^3},
 \qquad
 \frac{\abs{B(0,r)\cap W}}{\abs{B(0,r)}}\asymp(r/L)^3.
\end{equation}
For the lower bound one may integrate only over $r/4\le t\le r/2$; the upper bound
uses the cone cross-section directly.  Let
$U_\eta=W\cap B(0,\eta)$ with $0<\eta\ll r$.  Because the apex belongs to
$U_\eta$, the full neighborhood $N_r(U_\eta)$ contains $B(0,r)$.  In contrast,
\[
 W\cap N_r(U_\eta)\subset W\cap B(0,r+\eta),
\]
and the same integration gives an $O((r/L)^3)$ fraction of an $r$-ball.  Thus no
aspect-independent comparison can replace the coarse shading in the external buffer
$N_r(W)$ by $W\cap N_r(U_\eta)$.  Lemma~\ref{lem:buffer} provides the alternative
used in Theorem~\ref{thm:main}: the external buffer has comparable parent volume and
stable Katz--Tao density.

\subsection{Indexed multiplicity under coincident geometry}

Let there be $R$ different parent indices with identical geometric parent $W$, and
assign to each one a single identical child and identical nonnull shading.  The global
fine multiplicity is $R$, the inner multiplicity in every assigned group is $1$, and
the indexed coarse multiplicity is $R$.  Hence the multiplicity product is exact up to
constants.  If the coarse parents were geometrically deduplicated without weights,
the coarse multiplicity would become $1$ and the product would fail by the arbitrary
factor $R$.  This example forces indexed counting in
\eqref{eq:main-coarse-mult}, \eqref{eq:main-product}, and
\eqref{eq:main-kt}.

The same mechanism appears in Theorem~\ref{thm:icu3}.  If $Q_i=Q$ and $E_i=E$ for
all $N$ indices, then the indexed Katz--Tao density is $m=N$, whose inverse cancels
the indexed body mass $N\abs Q$.  This balance is the reason the passage from the
essentially-distinct theorem to the indexed theorem is carried out inside the paper.

\subsection{Mean-normalized density conversion}

Let $W=[0,1]^4$, fix $h=c_0a_1(W)$, include $V_0=W$, and partition $W$ into
$N=m^4$ congruent cubes $V_1,\ldots,V_N$ of side $m^{-1}$.  The indexed family
has bounded relative Frostman constant: its total body density in $W$ is $2$, a
proper convex test body cannot contain $V_0$, and the disjoint small cubes contained
in such a test body have total volume at most the test volume.

For large $m$,
\[
 r_0\asymp1,\qquad r_j\asymp N\quad(1\le j\le N).
\]
Since the large child and all small children each contribute total body mass $1$,
\[
 r_W^-\asymp1,\qquad
 \overline r_W
 =\frac{r_0\abs W+\sum_{j=1}^Nr_j\abs{V_j}}{2\abs W}
 \asymp N,\qquad
 r_W^+\asymp N.
\]
Shade only $V_0$.  Then
\[
 d_W=\frac12,\qquad
 d_{W,h}^{\rm wt}
 =\frac{r_0\abs W}
 {r_0\abs W+\sum_{j=1}^Nr_j\abs{V_j}}
 \asymp N^{-1}
 \asymp\frac{r_W^-}{\overline r_W}d_W.
\]
Thus \eqref{eq:weighted-unweighted-density} is sharp up to constants.  The elementary
local efficiency is
\[
 \eta_W^{(K)}
 =\frac{\overline r_W}{r_W^+}
  \left(\frac{r_W^-}{\overline r_W}\right)^K
 \asymp N^{-K}.
\]
Thus the example saturates the mean-normalized conversion without producing an
additional outer inflation-ratio factor.

\subsection{Load imbalance and the profile coefficient}

Take $P$ pairwise disjoint congruent parents with probability weights $p_j=1/P$.
Let $D_1=R$, $D_j=1$ for $j\ge2$, and $\alpha_j=1$.  Concentrate the fine
shaded density on the first parent.  When $R\gg P$, the global fine density can be
close to $1$, but a coarse denominator frozen over all $P$ parents has density at
most $O(P^{-1})$ if only the first parent carries substantial coarse shading.
Algebraically,
\[
 \mathfrak P_K
 =
 \frac{((R+P-1)/P)^K}
 {((R^{K/(K-1)}+P-1)/P)^{K-1}}
 \asymp P^{-1}\qquad(R\gg P).
\]
The displayed profile therefore decays like $P^{-1}$ in this regime.  The
comparable-scale corollary first selects one load level and consequently uses a
different coarse normalization.  By
Remark~\ref{rem:profile-equality}, the profile is optimal for the scalar H\"older
step; this example does not address optimality of the remaining geometric losses.

\subsection{Choice of projection direction}

Let $V=W$ be a box with side lengths
$w_1\ll w_2\le w_3\ll w_4$, take $h=w_1$, and project along the longest direction.
Then
\[
 \frac{\abs{N_h(V)}}{\abs{\pi N_h(V)}}\asymp w_4,
\]
not $h$.  The upper estimate in Lemma~\ref{lem:pv} contains the width in the
projection direction.  Choosing a shortest John direction turns that width into
$O(h)$ while still permitting arbitrarily tilted children.

\subsection{Affine orientation and cell-mass selection}

The affine parameterization in Section~\ref{sec:indexed} is needed even when all
boxes have the same center and long axis.  For
\[
 R_\phi=\operatorname{Rot}_{e_3}(\phi)
 ([-\delta,\delta]^2\times[-1,1]),\qquad
 \phi=j\delta,\quad0\le\phi\le\pi/4,
\]
all boxes have the same center and long axis, and any two square cross-sections overlap
in more than half their area.  The non-essential-distinctness graph contains a clique
of size comparable to $\delta^{-1}$.  Lemma~\ref{lem:canonical-boxes} avoids this
aspect loss by measuring relative affine transforms modulo the finite symmetry group
of the box.

Finally, arbitrary measurable shadings can place masses $2^{-2^k}$ in successive
grid cells.  Their positive dyadic levels are not bounded by a logarithm of the family
cardinality or aspect ratio.  The relative cutoff \eqref{eq:relative-cell-cutoff}
removes less than half the current mass and leaves a mass ratio at most $2J$; this is
why the fourth profile selection costs $O(\log(2+J))$ and remains valid for
densities smaller than every fixed scale power.

Finally, the product construction in Proposition~\ref{prop:lift-product} shows that
the codimension-one lift preserves, rather than improves, both the factor $H$ and the
density response inherited from the lower-dimensional theorem.


\appendix
\section{Dependence of constants and logarithmic losses}\label{app:parameters}

Theorem~\ref{thm:main} collects several logarithmic factors into one exponent
$A_\eps$.  We record their origins here and state the uniformity of the remaining
constants.

\subsection{Local analytic losses}

For a local assigned family of cardinality $N_W$, the indexed collision reduction
contributes
\[
 L_{N_W}^{-1}=(1+\lceil\log_2N_W\rceil)^{-1}.
\]
The three ordered projected John axes range between $h$ and $R$, so their dyadic
binning contributes
\[
 J_h^{-1}\gtrsim(1+\log_+(R/h))^{-3}.
\]
Since $N_W\le\#\cV$ and, uniformly in $W$,
$R/h\asymp_{\Lambda_W}w_4/w_1$, their product is bounded below by a fixed negative
power of
\[
 \Lfac=2+\log(2+\#\cV)+\log(2+\#\cW)
          +\log(2+w_4/w_1).
\]
Thus the cardinality and axis-binning losses remain in $\Lfac$, separately from the
power $(w_1/w_4)^\eps$.

\subsection{Cellular and profile losses}

The four mass-weighted selections in the cellular theorem have at most
\[
 O(\log(2+\#\cV)),\quad O(\log(2+\#\cV)),\quad
 O(\log(2+\#\cW)),\quad O(\log(2+J))
\]
choices, respectively.  They leave all parent loads in the H\"older profile.  The
cell count satisfies
\[
 J\lesssim\#\cW\prod_{j=1}^4(1+w_j/w_1),
\]
and therefore
\[
 \log(2+J)\lesssim1+\log(2+\#\cW)+4\log(2+w_4/w_1)
 \lesssim\Lfac.
\]
The retained mass is thus $\gtrsim\Lfac^{-4}$ times the incoming mass.  The
profile H\"older normalization raises this fraction to $K_\eps$, giving
$\Lfac^{-4K_\eps}$ in the coarse density.  The comparable-scale specialization
uses one additional parent-load selection.  One admissible exponent covering both
the profile theorem and that specialization is
\[
 A_\eps\ge A_{0,\eps}+5K_\eps+6.
\]
For the profile theorem itself, the four cellular selections account for the global
selection loss.

\subsection{Base-family profile quantities}

For each parent in the initial positive-mass support, the following quantities are
determined before the cellular refinement:
\[
 T_W,\quad D_W,\quad r_W^-,\quad\overline r_W,\quad r_W^+,\quad
 C^{\rm wt}_{F,h},\quad\kappa_W,\quad\eta_W.
\]
The exponent $K_\eps\ge2$ comes from the three-dimensional response and is
unchanged by the lift.  The parent probabilities $p_W$, the profile
$\mathfrak P_\eps$, and the fine density use the denominators fixed in
Section~\ref{sec:prelim}.  The data profile therefore remains displayed separately
from $\Lfac^{-A_\eps}$ and the implicit constant.

\subsection{Uniformity in the nominal parent scale}

The proof uses the common analytic radius $h=w_1/100$, whereas
Theorem~\ref{thm:r4-local} is written with $h=c_0a_1(W)$.  The parent comparison
hypothesis gives
\[
 c_0=\frac{h}{a_1(W)}\in
 \left[\frac1{100\Lambda_W},\frac{\Lambda_W}{100}\right].
\]
Every geometric constant in Lemmas~\ref{lem:pv}--\ref{lem:fr-inheritance} is uniform
on this compact interval.  Also
\[
 \frac{a_1(W)}{a_4(W)}\asymp_{\Lambda_W}\frac{w_1}{w_4},
 \qquad w_1\le\Lambda_Wa_1(W),
\]
which justifies both the common aspect factor and the common buffer radius.

The profile theorem's implicit constants may depend on $\eps$, the dimension, and
$\Lambda_W$; the comparable-scale specialization may additionally depend on
$\Lambda_V$.  They are uniform in $\lambda$, the family cardinalities, duplicate
multiplicity, child orientation, child aspect ratio, and child minimum scale.  The
explicit profiles and logarithms retain the corresponding base-family data.

\section{Derivation of the three-dimensional input}\label{app:external}

Proposition~\ref{prop:ed-cu3} combines several statements from Wang and Zahl with a
short localization and density-threshold argument.  This appendix gives the notation
dictionary for that derivation and distinguishes the source inputs from the indexed
extension proved in Section~\ref{sec:indexed}.

\medskip
\noindent\textit{$a,b$ (manuscript) $\longleftrightarrow a,b$ (source).}
These are the normalized John dimensions in Definitions~3.1 and 5.4; here
$a=q_1/q_3$ and $b=q_2/q_3$.

\smallskip
\noindent\textit{$N,v$ $\longleftrightarrow \#\mathcal P,\abs P$.}
These are the cardinality and comparable body-volume scale in (5.5), with
$v\asymp ab$ after normalizing the longest dimension to one.

\smallskip
\noindent\textit{$\rho$ $\longleftrightarrow$ the dense-shading parameter.}
Definition~3.1 uses the same average-mass notion,
$\sum_P\abs{Y(P)}\ge\rho\sum_P\abs P$, and Definition~5.4 assumes
$\rho\ge a^\eta$.

\smallskip
\noindent\textit{$m=\Delta_{\max}(\cQ)\longleftrightarrow
C_{\mathrm{KT-CW}}(\mathcal P)$.}
The source locations are Definition~1.3$'$(A), equation~(4.1), and Remark~4.2(A).
The quantities agree after identifying $\mathcal P=\cQ$; the source family at this
stage is a set, not our later indexed multiset.

\smallskip
\noindent\textit{$\ell$ $\longleftrightarrow C_{\mathrm{F-SW}}(\mathcal P)$.}
The source locations are Definition~1.3$'$(B), equation~(4.2), and Remark~4.2(A),
with the test collection restricted to slabs. This auxiliary slab constant is not the
relative convex Frostman constant $C_F(\cQ,B')$ used later in the paper.

\smallskip
\noindent\textit{$\mathcal D(\cQ)$ $\longleftrightarrow D(\mathcal P)$.}
The source locations are Definition~5.4, equation~(5.6), and Remark~5.6. The
calligraphic letter here avoids collision with the manuscript's other density symbols.

\smallskip
\noindent\textit{$t,t$ $\longleftrightarrow \sigma,\omega$.}
These are the positive parameters in $F(\sigma,\omega)$, Definition~5.4. In
Proposition~\ref{prop:ed-cu3} both are chosen as $t=\eps/10$, while the source's
displayed $a^\varepsilon$ loss is invoked with $\varepsilon=\eps/2$; thus the source
threshold exponent is $\eta_{\mathrm{WZ}}(\eps/2;t,t)$.
\medskip

With this dictionary, (5.5) agrees with \eqref{eq:wz-positive} up to the fixed John
comparability constants and the source constant $\kappa$. The facts
$Nv\lesssim m$, $m\gtrsim1$, and $\ell\lesssim a^{-1}$ come, respectively, from
the fixed ambient localization, testing the Katz--Tao constant on one member, and the
elementary slab-volume test.  For every $P$, subfamily monotonicity gives
$C_{\mathrm{KT-CW}}(\mathcal P[N_b(P)])\le m$, so Remark~5.6, Situation~2, gives
$\mathcal D\lesssim m^{1/2}$.  Hence, writing
$B=m^{-3/2}\ell Nv^{1/2}$,
\[
 B\mathcal D\lesssim m^{-1}\ell Nv^{1/2}
 \lesssim\ell v^{-1/2}\lesssim a^{-2}.
\]
Thus the positive source parameters are absorbed into $a^\eps$ by the derivation
following \eqref{eq:wz-positive}, yielding
\eqref{eq:ed-cu3}. The branch $0<\rho<a^\eta$ uses the
largest shading and is internal to Proposition~\ref{prop:ed-cu3}.  The source chain
begins at $D(0,0)$, but the convex assertion used in the union estimate is
$F(t,t)$ with $t=\eps/10>0$.

The relevant source result chain is precise. Theorem~1.9 gives $D(0,0)$.
At $\sigma=\omega=t$, Definition~1.5 makes $D(t,t)$ weaker by the factor
$(\#\mathcal T)^{-t}\le1$. Proposition~1.6 gives $E(t,t)$, and
Proposition~5.14 gives $F(t,t)$ for $0<t\le2/3$
\cite{wangzahl2025convexunions}. Remark~5.6, Situation~2, supplies the separate
bound for the geometric factor $D$ used above.

\begin{enumerate}
\item Wang and Zahl's three-dimensional convex assertion $F(\sigma,\omega)$ is used
for positive $\sigma$ and $\omega$, together with their endpoint theorem and
equivalence result \cite{wangzahl2025convexunions}.  The source assertion assumes an
essentially-distinct family and a positive density threshold.  Proposition
\ref{prop:ed-cu3} converts it to the required all-density form by
positive-parameter absorption and the separate largest-shading argument.

\item The source definition of essential distinctness is
\eqref{eq:essentially-distinct}, so coincident indexed bodies lie outside its scope.
Lemma~\ref{lem:canonical-boxes} and Theorem~\ref{thm:icu3} provide the indexed
extension, with the explicit loss $L_N^{-1}$.

\item Guth, Wang, and Zahl formulate the three-dimensional factoring architecture and
develop its induced-shading density mechanism \cite{guthwangzahl2026streamlined}.
Subsequent analysis of the printed non-cellular transitions gives a conditional
cellular replacement for the relevant compatibility interface
\cite{huayang2026cellularrepair}.  These works motivate the factoring architecture;
the four-dimensional theorem uses the self-contained cellular argument of Section
\ref{sec:cellular}.

\item The sticky works of Wang--Zahl and Choudhuri contain compatible coarse/fine
refinements and multiplicity products
\cite{wangzahl2026sticky,choudhuri2024stickyR4}.  They provide context for the
coarse/fine organization, while Theorem~\ref{thm:main} is proved without their
extremal or sticky hypotheses.

\item The four-dimensional maximal theorem of Borges, Chan, Chen, Liu, Xi, and Zhan
\cite{borgesetal2025restrictionR4} locates the structural result within the broader
four-dimensional problem.  The factoring proof itself is independent of its
restriction, two-ends, plany, and transversality conclusions.
\end{enumerate}

The division of roles is therefore explicit: $F(t,t)$ is used at positive parameters,
the small-density branch is proved in Proposition~\ref{prop:ed-cu3}, and the passage
from essentially-distinct bodies to indexed families is carried out by Theorem
\ref{thm:icu3}.

\section{Supplementary boundary cases}\label{app:counterexamples}

The main text emphasizes the mechanisms that drive the estimates.  The examples here
clarify three set-theoretic and normalization conventions used by the common
refinement.

\subsection{Coarse shadings and fine fibres}

Let a parent have positive fine shading in a small measurable subset of a selected
cell $Q$.  By definition, $Z(W^\sharp)$ contains the whole cell.  At a point
$x\in Q$ outside the fine shading one has $x\in Z(W^\sharp)$ but
$(\cV_W)_{Y'}(x)=\varnothing$.  Thus \eqref{eq:main-fibre} is valid because empty
fibres contribute nothing, but the stronger parent-set identity
\[
 \{W:(\cV_W)_{Y'}(x)\ne\varnothing\}
 =\{W:x\in Z(W^\sharp)\}
\]
is false in general.

\subsection{The global null cleanup}

Start with several parents having positive shading in each of several cells.  Add one
parent with positive shading in one such cell and with nonempty null dust in the other
cells.  Positive-measure activity counts the added parent only in its positive cell,
but the set-theoretic neighborhood of its union sees the dust in all cells.  Without
deleting the dust, neither the pointwise implication from fine contribution to active
parent nor the grid-neighbor comparison \eqref{eq:grid-neighbor} is valid everywhere.

Deleting null pieces separately for each parent is insufficient at points shared by
different parents, because it can change the global count while leaving the point in
the global union.  The single set $E$ in \eqref{eq:null-cleanup} is a finite union of
null parent--cell pieces and is removed from every child.  It preserves every mass and
all previously selected pointwise multiplicities.

\subsection{Limits of unweighted Frostman inheritance}

Let $W=[0,1]^4$, include $V_0=W$, and let
\[
 V_j=I_j\times[0,1]^3,\qquad1\le j\le N,
\]
where the consecutive disjoint intervals $I_j$ have length $\eta=N^{-4}$.  The
original indexed family has bounded relative Frostman constant: a proper convex test
body cannot contain $V_0$, while the small slabs it contains have disjoint interiors
and total volume at most that of their convex hull.

Take $h=N^{-2}$ and $P_j=N_h(V_j)$.  The convex body
\[
 K=N_h([0,N\eta]\times[0,1]^3)
\]
contains $P_1,\ldots,P_N$ but not $P_0$.  Since $N\eta\ll h$,
\[
 \abs K\asymp h,\qquad \abs{P_j}\asymp h,\qquad
 \frac{\sum_{j=1}^N\abs{P_j}}{\abs K}\asymp N.
\]
Meanwhile the total thickened-body density in a fixed containing box is bounded.
Thus the relative Frostman constant after thickening can grow like $N$, showing
that an inflation-sensitive coefficient is unavoidable; that geometric calculation
alone does not identify the sharp universal ratio.

The intrinsic quantity for the transfer is the exact weighted coefficient
$C^{\rm wt}_{F,h}$ from \eqref{eq:weighted-frostman}.  The universal comparison
\[
 C^{\rm wt}_{F,h}\le C\,r_+/\overline r
\]
is the proof-valid upper bound of
Lemma~\ref{lem:weighted-frostman-comparison}.  Here the relevant scales are
\[
 r_0\asymp1,\qquad r_j\asymp h/\eta=N^2,\qquad
 r^-\asymp\overline r\asymp1,\qquad r^+\asymp N^2.
\]
Indeed, the small slabs have total weighted mass
$\sum_{j=1}^Nr_j\abs{V_j}\asymp Nh=N^{-1}$, whereas the large child contributes
mass comparable to $1$.  Testing the weighted coefficient on the unthickened convex
strip $[0,N\eta]\times[0,1]^3$ gives weighted density comparable to $N^2$.
This example isolates the need for an inflation-sensitive coefficient; it does not
assert sharpness of the full inheritance or profile theorem.  Its radius is not tied
to the shortest scale of the unit parent.  The scale-compatible construction in
Section~\ref{sec:sharpness} separately treats the mean-normalized density conversion
when $h\asymp a_1(W)$.


{\small
\bibliographystyle{amsplain}
\bibliography{references}
}

\end{document}